\documentclass[11pt]{amsart}
  \RequirePackage{amsmath}
\RequirePackage{amssymb} \RequirePackage{amsthm}
\RequirePackage{epsfig} 
\def\ep{\varepsilon}
\newcommand{\hg}{\mathcal{H}_{g}}
\newcommand{\D}{\mathbb{D}}
\newcommand{\M}{\mathcal{M}}

\newcommand{\T}{\mathbb{T}}
\newcommand{\N}{\mathbb{N}}

\renewcommand{\H}{\mathcal{H}}

\newcommand{\vp}{\vp}

\newcommand{\og}{\mathrm{O}}
\newcommand{\op}{\mathrm{o}}

\def\a{\alpha}       \def\b{\beta}        \def\g{\gamma}
       \def\De{{\Delta}}    \def\e{\varepsilon} \def\Denw{{\Delta}^{\om}_{n}}
     \def\om{\omega}      
       \def\t{\theta}       
\def\p{\psi}         \def\r{\rho}         
                  \def\vp{\varphi}

 \def\vr{\varrho}
\def\BMOA{\mathord{\rm BMOA}}

\def\R{{\mathcal R}}
\def\I{{\mathcal I}}

\def\C{{\mathbb C}}

\DeclareMathOperator{\supp}{supp}
\newcommand{\h}{\mathcal{H}}
\newcommand{\hti}{\widetilde{\mathcal{H}}}
\newcommand{\ho}{\widehat{\omega}}
\newcommand{\Mp}{{\mathcal M}_p}

\newtheorem{theorem}{Theorem}
\newtheorem{lemma}[theorem]{Lemma}
\newtheorem{proposition}[theorem]{Proposition}

\newtheorem{corollary}[theorem]{Corollary}

\newtheorem{lettertheorem}{Theorem}
\newtheorem{lettercorollary}[lettertheorem]{Corollary}
\newtheorem{letterlemma}[lettertheorem]{Lemma}

\theoremstyle{definition}

\theoremstyle{remark}

\theoremstyle{remarks}

\begin{document}
\title[Generalized Hilbert operators on weighted Bergman spaces] {Generalized Hilbert operators on weighted Bergman spaces}

\author{Jos\'e \'Angel Pel\'aez}

\address{Departamento de An´alisis Matem´atico, Universidad de M´alaga, Campus de
Teatinos, 29071 M´alaga, Spain} \email{japelaez@uma.es}

\thanks{This research was supported in part by the Ram\'on y Cajal program
of MICINN (Spain), Ministerio de Edu\-ca\-ci\'on y Ciencia, Spain,
(MTM2011-25502), from La Junta de Andaluc{\'i}a, (FQM210) and
(P09-FQM-4468), MICINN- Spain ref.
MTM2011-26538}

\author{Jouni R\"atty\"a}
\address{University of Eastern Finland, P.O.Box 111, 80101 Joensuu, Finland}
\email{jouni.rattya@uef.fi}

\date{\today}



\begin{abstract}
The main purpose of this paper is to study the generalized Hilbert
operator
    \begin{equation*}
    \mathcal{H}_g(f)(z)=\int_0^1f(t)g'(tz)\,dt
    \end{equation*}
acting on the weighted Bergman space $A^p_\om$, where the weight
function $\om$ belongs to the class $\R$ of regular radial weights
and satisfies the Muckenhoupt type condition
    \begin{equation}\label{Mpconditionaabstract}
    \sup_{0\le r<1}\bigg(\int_{r}^1\left(\int_t^1\om(s)ds\right)^{-\frac{p'}{p}}\,dt\bigg)^\frac{p}{p'}
    \int_{0}^r(1-t)^{-p}\left(\int_t^1\om(s)ds\right)\,dt<\infty. \tag{\dag}
    \end{equation}
If $q=p$, the condition on $g$ that characterizes the boundedness
(or the compactness) of $\hg: A^p_\om\to A^q_\om$ depends on $p$
only, but the situation is completely different in the case $q\ne
p$ in which the inducing weight $\om$ plays a crucial role. The
results obtained also reveal a natural connection to the
Muckenhoupt type condition \eqref{Mpconditionaabstract}. Indeed,
it is shown that the classical Hilbert operator (the case
$g(z)=\log\frac{1}{1-z}$ of $\H_g$) is bounded from
$L^p_{\int_{t}^1\om(s)\,ds}\left([0,1)\right)$ (the natural
restriction of $A^p_\om$ to functions defined on $[0,1)$) to
$A^p_\om$ if and only if $\om$ satisfies the condition
\eqref{Mpconditionaabstract}. On the way to these results
decomposition norms for the weighted Bergman space $A^p_\om$ are
established.

\end{abstract}
\maketitle

\section{Introduction}\label{sec:lacunary}

Let $\H(\D)$ denote the space of all analytic functions in the
unit disc $\D$ of the complex plane $\C$. A function $\omega:\D\to
(0,\infty)$, integrable over $\D$, is called a \emph{weight
function}\index{weight function} or simply a
\emph{weight}\index{weight}. It is \emph{radial}\index{radial
weight} if $\omega(z)=\omega(|z|)$ for all $z\in\D$. For
$0<p<\infty$ and a weight $\omega$, the \emph{weighted Bergman
space} \index{weighted Bergman space}
$A^p_\omega$\index{$A^p_\om$} consists of those $f\in\H(\D)$ for
which
    $$
    \|f\|_{A^p_\omega}^p=\int_\D|f(z)|^p\omega(z)\,dA(z)<\infty,\index{$\Vert\cdot\Vert_{A^p_\omega}$}
    $$
where $dA(z)=\frac{dx\,dy}{\pi}$ \index{$dA(z)$}is the normalized
Lebesgue area measure on $\D$. As usual, we write
$A^p_\alpha$\index{$A^p_\alpha$} for the classical weighted
Bergman space induced by the standard radial weight
$\omega(z)=(1-|z|^2)^\alpha$ with $-1<\alpha<\infty$.

Every $g\in\H(\D)$ induces the \emph{generalized Hilbert operator}
$\hg$, defined by
    \begin{equation}\label{H-g}
    \mathcal{H}_g(f)(z)=\int_0^1f(t)g'(tz)\,dt, \quad f\in\H(\D).
    \end{equation}
The sharp condition
    \begin{equation}\label{99}
    \int_{0}^1\left(\int_t^1\om(s)ds\right)^{-\frac{1}{p-1}}dt<\infty,\quad
    p>1,
    \end{equation}
ensures that the integral in \eqref{H-g} defines an analytic
function on $\mathbb{D}$ for each $f\in A^p_\om$. The choice
$g(z)=\log\frac{1}{1-z}$ in \eqref{H-g} gives an integral
representation of the \emph{classical Hilbert operator} $\h$. The
Hilbert operator $\h$ is a prototype of a Hankel operator which
has attracted a considerable amount of attention during the last
years in operator theory on spaces of analytic functions.
Questions related to the boundedness, the operator norm and the
spectrum of $\h$ have been studied in \cite{AlMonSa,DiS,Di,DJV}.
These studies reveal a natural connection from $\h$ to the
weighted composition operators, the Szeg\"{o} projection and the
Legendre functions of the first kind. For further information on
$\H$, the reader is invited to see the recent monograph
\cite[Chapter~13]{ArJeVu}.

The primary purpose of this paper is to study the operator $\hg$
acting on the weighted Bergman space $A^p_\om$ induced by a radial
weight $\om$. We are particularly interested in basic properties
such as the question of when $\hg:A^p_\om\to A^q_\om$ is bounded
or compact.

As far as we know, the generalized Hilbert operator $\hg$ has not
been extensively studied in the existing literature. The operator
was introduced recently in \cite{GaGiPeSis}, where it was shown,
among other things, that the membership of the analytic symbol $g$
to the mean Lipschitz space $\Lambda\left(p,\frac{1}{p}\right)$
characterizes the boundedness of $\hg$ on the Bergman space
$A^p_\alpha$ ($-1<p-2<\a<\infty$), on the Hardy space $H^p$
($1<p\le 2$) and also on certain Dirichlet type spaces. The proofs
of these results are based on the identity $\left(\hg\right)'(f)=
\mathcal{H}_{g'}(zf)$ together with properties of the maximum
modulus and the smoothness of the sequence of moments
$\left\{\int_{0}^1t^k |f(t)|\,dt\right\}_{k=0}^\infty$ of
functions in these spaces.

The approach we take to the study of $\hg$ allows us to determine
those analytic symbols $g$ for which $\hg: A^p_\om \to A^q_\om$,\,
$1<p,q<\infty$, is bounded or compact, provided the regular weight
$\om$ (see Section~\ref{mainresults} for the definition) satisfies
the Muckenhoupt type condition
    \begin{equation}\label{Mpcondition}
    \sup_{0\le r<1}  \bigg(\int_{r}^1\left(\int_t^1\om(s)ds\right)^{-\frac{p'}{p}}\,dt\bigg)^\frac{p}{p'}
    \int_{0}^r(1-t)^{-p}\left(\int_t^1\om(s)ds\right)\,dt<\infty.
    \end{equation}
By the classical results of Muckenhoupt~\cite{Muckenhoupt1972},
the condition \eqref{Mpcondition} characterizes those real
functions $v\in[0,1)\to(0,\infty)$ for which the Hardy type
operator $\int_t^1\frac{h(s)}{1-ts}\,ds$ is bounded on
$L^p_{\int_t^1|v(s)|\,ds} $. We mention that each
standard radial weight $\om(r)=(1-r^2)^\a$, $-1<p-2<\a<\infty$, is
regular and satisfies \eqref{Mpcondition}. Our results show the
interesting phenomenon that when the inducing powers of the domain
and the target spaces are equal, i.~e. $q=p$, then the weight
function $\om$ does not play any role in the condition on $g$ that
characterizes the boundedness (or the compactness) of $\hg:
A^p_\om \to A^q_\om$, although the description depends on $p$. However,  the situation is completely different in
the case $q\ne p$ in which the inducing weight $\om$ plays a
crucial role.

\section{Preliminaries and main results}\label{mainresults}

For $0<p\le\infty$, the \emph{Hardy space} $H^p$ consists of those
$f\in\H(\D)$ for which
    \begin{equation*}
    \|f\|_{H^p}=\lim_{r\to1^-}M_p(r,f)<\infty,\index{$\Vert\cdot\Vert_{H^p}$}
    \end{equation*}
where
    $$
    M_p(r,f)=\left (\frac{1}{2\pi }\int_0^{2\pi}
    |f(re^{i\theta})|^p\,d\theta\right )^{\frac{1}{p}},\quad 0<p<\infty,
    $$
 and
    $$
    M_\infty(r,f)=\max_{0\le\theta\le2\pi}|f(re^{i\theta})|.
    $$
Throughout the paper, the letter $C=C(\cdot)$ will denote a
constant whose value depends on the parameters indicated in the
parenthesis, and may change from one occurrence to another. We
will use the notation $a\lesssim b$ if there exists $C=C(\cdot)>0$
such that $a\le Cb$, and $a\gtrsim b$ is understood in an
analogous manner. In particular, if $a\lesssim b$ and $a\gtrsim
b$, then we will write $a\asymp b$.

The \emph{distortion function} of a radial weight
$\om:[0,1)\to(0,\infty)$ is defined by
    $$
    \psi_{\om}(r)=\frac{1}{\om(r)}\int_{r}^1\om(s)\,ds,\quad
    0\le r<1.
    $$
It was introduced in~\cite{Si} on the way to the Littlewood-Paley
formulas for weighted Bergman spaces. A radial weight $\om$ is
called {\em{regular}}, if it is continuous and its distortion
function satisfies
    \begin{equation}\label{eq:r0}
    \psi_\om(r)\asymp(1-r),\quad 0\le r<1.
    \end{equation}
The class of all regular weights is denoted by $\R$. For basic
properties and concrete examples of regular weights, see
\cite[Chapter~1]{PelRat} and \cite{Si}, and references therein. At
this point we settle to mention that each standard weight
$\om(r)=(1-r^2)^\a$ with $-1<\a<\infty$ is regular. From now on we
will use the notation $\widehat{\om}(r)=\int_r^1\om(s)\,ds$ so
that \eqref{eq:r0} ensures $\widehat{\om}(r)\asymp\om(r)(1-r)$ for
$\om\in\R$. Moreover, for each radial weight $\om$ we will write
$\om_\g(r)=(1-r)^\g\om(r)$,\, $-\infty<\g<\infty$, and $\om\in
\Mp$ if $\om$ satisfies the Muckenhoupt type condition
\eqref{Mpcondition}. We will write $\|T\|_{(X,Y)}$ for the norm of
an operator $T:X\to Y$, and if no confusion arises with regards to
$X$ and $Y$, we will simply write $\|T\|$.

Our study of $\hg$ on weighted Bergman spaces leads us to consider
other classes of weighted spaces. For $0<p\le \infty$,
$0<q<\infty$, $0\le \gamma<\infty$ and a radial weight~$\om$, the
\emph{mixed norm space} $H(p,q,\om_\g)$ consists of those
$g\in\H(\D)$ such that
    \begin{equation*}
    \left\|g\right\|^q_{H(p,q,\om_\g)}=\int_0^1 M^q_p(r,g)(1-r)^{\gamma}\om(r)\,dr<\infty.
    \end{equation*}
Moreover, if in addition $-\infty<\b<\infty$, we will denote $g\in
H(p,\infty,(\widehat{\om}^\b)_\g)$, whenever
    \begin{equation*}
    \left\|g\right\|^q_{H(p,\infty,(\widehat{\om}^\b)_\g)}=\sup_{0<r<1}M_p(r,g)(1-r)^\g \ho(r)^\b<\infty.
    \end{equation*}
We will simply write $H(p,q,\om)$ and
$H(p,\infty,\widehat{\om}^\b)$ if $\g=0$. It is clear that
$H(p,p,\om)=A^p_\om$. The mixed norm spaces play an essential role
in the closely related question of studying the coefficient
multipliers on Hardy and weighted Bergman spaces~\cite{ArJeVu}.

For given $1\le p<\infty$, $0<\a\le 1$ and $0\le \b<\infty$, we
say that $g\in\H(\D)$ belongs to $\Lambda\left(p,\alpha,
\ho^\b\right)$, if $g'\in H(p,\infty,(\widehat{\om}^{-\b})_{1-\alpha})$, that
is,
     \begin{equation*}
    \left\|g\right\|_{\Lambda\left(p,\alpha, \ho^\b\right)}=\sup_{0<r<1}\frac{M_p(r,g')(1-r)^{1-\alpha}}{\ho(r)^\b}+|g(0)|<\infty.
    \end{equation*}
Since $0<\a\le1$ and $0\le\b<\infty$, we have
$\Lambda\left(p,\alpha, \ho^\b\right)\subset H^p$, and therefore
each function $g\in\Lambda\left(p,\alpha, \ho^\b\right)$ has a
non-tangential limit $g(e^{i\theta})$ almost everywhere on the
unit circle $\T$. Indeed, if $\b=0$, then
$\Lambda\left(p,\alpha,\ho^\b\right)$ is nothing else but the
\emph{mean Lipschitz space} $\Lambda\left(p,\alpha\right)$ that
consists of those $g\in \H(\D)$ having non-tangential limits
$g(e^{i\theta})$ almost everywhere and for which
    $$
    \sup_{0<h\le t}\left(\int_0^{2\pi}
    |g(e^{i(\theta+h)})-g(e^{i\theta})|^p \frac{d\theta}{2\pi}\right)^{1/p}=\og(t^{\alpha}), \quad t\to
    0,
    $$
see a classical result of Hardy and
Littlewood~\cite[Theorem~5.4]{Duren1970}.

We will see that if $1<p<\infty$ and $\om\in\R\cap\Mp$, then
$\H_g:A^p_\om\to A^p_\om$ is bounded if and only if $g\in\Lambda
(p,\frac{1}{p})$. The spaces $\Lambda (p,\frac{1}{p})$ form a
nested scale contained in $\BMOA$ \cite{BSS}:
    $$
    \Lambda \left(q,\frac{1}{q}\right)\subset \Lambda  \left(p,\frac{1}{p}\right)\subset\BMOA,\quad 1\le q<p<\infty .
    $$
The absence of $\om$ in the condition on $g$ that characterizes
the boundedness does not come as a surprise in view of
\cite[Theorem $3$]{GaGiPeSis}. However, the situation is
completely different in the case $q\ne p$ in which the inducing
weight $\om$ plays a crucial role. In particular, if $q>p$, then
the space
$\Lambda\left(q,\frac{1}{p},\ho^{\frac{1}{p}-\frac{1}{q}}\right)$,
that can be described also by a growth condition on the modulus of
continuity of order $q$ of $g(e^{i\t})$ by
Proposition~\ref{pr:Lambdafrontera} below, comes naturally to the
picture.

Our main result on $\H_g$ reads as follows.

\begin{theorem}\label{mainth}
Let $1<p,q<\infty$, $\om\in\R\cap\Mp$ and $g\in \H(\D)$.
    \begin{enumerate}
    \item[\rm(i)] If $1<p\le q<\infty$, then $\hg: A^p_\om \to A^q_\om$ is bounded if and only if $g\in\Lambda\left(q,\frac{1}{p}, \ho^{\frac{1}{p}-\frac{1}{q}}\right)$.
Moreover, if $g\in \Lambda\left(q,\frac{1}{p},
\ho^{\frac{1}{p}-\frac{1}{q}}\right)$, then
    $$
    \|\hg\|_{\left(A^p_\om, A^q_\om\right)}\asymp\|g-g(0)\|_{\Lambda\left(q, \frac{1}{p}, \ho^{\frac{1}{p}-\frac{1}{q}}\right)}.
    $$
    \item[\rm(ii)] If $1<q<p<\infty$, then $\H_g:A^p_\om\to A^q_\om$ is
bounded if and only if $g'\in
H\left(q,s,\ho_{s\left(1-\frac{1}{q}\right)}\right)$, where
$\frac{1}{q}-\frac{1}{p}=\frac{1}{s}$. Moreover, if $g'\in
H\left(q,s,\ho_{s\left(1-\frac{1}{q}\right)}\right)$, then
    $$
    \|\hg\|_{\left(A^p_\om, A^q_\om\right)}\asymp\|g'\|_{H\left(q,s,\ho_{s\left(1-\frac{1}{q}\right)}\right)}.
    $$
    \end{enumerate}
\end{theorem}

It is easy to see, by using the auxiliary result on $\Mp$, stated
as Lemma~\ref{Lemma:u_p} below, that the space
$\Lambda\left(q,\frac{1}{p},
\ho^{\frac{1}{p}-\frac{1}{q}}\right)$ is not trivial if
$\om\in\R\cap\Mp$ no matter how large $q$ is.

Our approach to the study of the boundedness of $\hg$ on weighted
Bergman spaces arises the Muckenhoupt type condition
\eqref{Mpcondition} in a natural way. In order to explain this
phenomenon better, we recall that the {\em{sublinear Hilbert
operator}} is defined by
    $$
    \hti(f)(z)=\int_0^1\frac{|f(t)|}{1-tz}\,dt.$$
We shall see that it behaves like a kind of maximal function for all generalized
Hilbert operators under the assumptions of Theorem~\ref{mainth}. Indeed,
 we will show that
\begin{equation}\label{justificationmaximal}
    \|\hg(f)\|_{A^q_\om}\lesssim \|f\|_{A^p_\om}+\|f\|^{s_1}_{A^p_\om}\|\hti(f)\|^{s_2}_{A^p_\om},\quad
    s_1+s_2=1.
    \end{equation}
This together with the sharp inequality
    \begin{equation}\label{le:minfty}
    \int_0^1
    M^p_\infty(r,f)\,\ho(r)\,dr\le\frac{\pi}{2}\|f\|^p_{A^p_\om},
    \end{equation}
which can be easily obtained by integrating the known inequality $
\int_0^s M^p_\infty(r,f)\,dr\le\pi s M_p^p(s,f)$~\cite{Po61/62},
lead us to consider the following result of interest of its
own.

\begin{theorem}\label{th:gorro}
Let $1<p<\infty$ and $\om\in\R$ such that \eqref{99} is satisfied.
Then the following assertions are equivalent:
    \begin{enumerate}
    \item[\rm(i)] $\h:L^p_{\ho}\to A^p_\om$ is bounded;
    \item[\rm(ii)] $\hti:L^p_{\ho}\to A^p_\om$ is bounded;
    \item[\rm(iii)] $\om$ satisfies the Muckenhoupt type condition
    \begin{equation}\label{eq:h1}
    \begin{split}
    \Mp(\om)&=\sup_{0\le
    r<1}\bigg(\int_{r}^1\widehat{\om}(t)^{-\frac{1}{p-1}}\,dt\bigg)^{1-\frac{1}{p}}
    \left(\int_{0}^r(1-t)^{-p}\,\widehat{\om}(t)\,dt\right)^\frac{1}{p}<\infty.
    \end{split}
    \end{equation}
\end{enumerate}
Moreover, if $\om\in\Mp$, then
    $$
    \|\h\|_{\left(L^p_{\ho},A^p_\om\right)}\asymp
    \|\hti\|_{\left(L^p_{\ho},A^p_\om\right)}\asymp\Mp(\om).
    $$
\end{theorem}

Theorem \ref{th:gorro} together with \eqref{le:minfty} extends
\cite[Theorem~1]{Di} and \cite[Theorem~5~(ii)]{GaGiPeSis}.

\begin{corollary}\label{hilbertapw}
Let $1<p<\infty$ and $\om\in\R\cap\Mp$. Then, both the Hilbert
operator~$\h$ and the sublinear Hilbert operator $\hti$ are
bounded on $A^p_\om$.
\end{corollary}

We also work partially with radial weights $\om$ for which the
quotient $\frac{\p_\om(r)}{1-r}$ is not bounded. More precisely,
we say that a radial weight $\om$ is \emph{rapidly increasing}, if
it is continuous and
    \begin{equation*}
    \lim_{r\to 1^-}\frac{\p_\om(r)}{1-r}=\infty.
    \end{equation*}
The class of rapidly increasing weights is denoted by $\I$. It is
easy to see that $A^p_\om\subset A^p_\beta$ for each $\om\in\I$
and for any $\beta>-1$, see \cite[Section~1.4]{PelRat}. Typical
examples of rapidly increasing weights are
    \begin{equation*}
    \om(r)=\left((1-r)\prod_{n=1}^{N}\log_n\frac{\exp_{n}0}{1-r}\left(\log_{N+1}\frac{\exp_{N+1}0}{1-r}\right)^\a\right)^{-1}
    \end{equation*}
for all $1<\a<\infty$ and $N\in\N=\{1,2,\ldots\}$. Here, as usual,
$\log_nx=\log(\log_{n-1}x)$, $\log_1x=\log x$, $\exp_n
x=\exp(\exp_{n-1}x)$ and $\exp_1x=e^x$.

The right choice of the norm used is in many cases a key tool
for a good understanding of how a concrete operator acts in a
given space. Here, an $l^p$-type norm of the Hardy norms of blocks
of the Maclaurin series, whose size depend on the weight $\om$,
provides us an effective skill to study the boundedness and
compactness of $\hg$ on weighted Bergman space $A^p_\om$. The size
of these blocks reflects the growth of the inducing weight $\om$.
We remind the reader that {\em{decomposition results}} have been an
important tool for the study of a good number of questions on
spaces of analytic function on $\D$. They have been applied, for
example, when studying coefficient multipliers~\cite{ArJeVu},
Carleson measures~\cite{GPP} and the generalized Hilbert
operator~\cite{GaGiPeSis}. The results proved by M.~Mateljevi\'c
and M.~Pavlovi\'c in~\cite{MatelPavstu} (see also~\cite{Pabook})
offer such a decomposition result on~$A^p_\om$ when $\om\in\R$, see also
\cite{Pav1,Pav2} for further results. This
because a calculation based on \cite[Lemma $1.1$]{PelRat} says
that \cite[Theorem~2.1 (b)]{MatelPavstu} works for $\om\in\R$,.
However, to the best of our knowledge, results in the existing
literature do not cover the less understood case of the class~$\I$
of rapidly increasing weights. Indeed, only some special cases
have been considered in~\cite[Theorem~6.1]{GPP}.

\par We will develop a
technique that allows us to give a unified treatment for both
classes~$\R$ and $\I$. Theorem~\ref{th:dec1} below is our main
result in that direction. To give the precise statement, we need
to introduce some notation. To do this, let $\om\in\I\cup\R$ such
that $\int_0^1 \om(r)\,dr=1$. For each $\a>0$ and
$n\in\N\cup\{0\}$, let $r_n=r_n(\om,\a)\in[0,1)$ be defined by
    \begin{equation}\label{rn}
    \widehat{\om}(r_n)=\int_{r_n}^1\om(r)\,dr=\frac{1}{2^{n\a}}.
    \end{equation}
Clearly, $\{r_n\}_{n=0}^\infty$ is an increasing sequence of
distinct points on $[0,1)$ such that $r_0=0$ and $r_n\to1^-$, as
$n\to\infty$. For $x\in[0,\infty)$, let $E(x)$ denote the integer
such that $E(x)\le x<E(x)+1$, and set
$M_n=E\left(\frac{1}{1-r_{n}}\right)$ for short. Write
    $$
    I(0)=I_{\om,\alpha}(0)=\left\{k\in\N\cup\{0\}:k<M_1\right\}
    $$
and
   \begin{equation*}
    I(n)=I_{\om,\a}(n)=\left\{k\in\N:M_n\le
    k<M_{n+1}\right\}
  \end{equation*}
for all $n\in\N$. If
$f(z)=\sum_{n=0}^\infty a_nz^n$ is analytic in~$\D,$ define the
polynomials $\Delta^{\om,\a}_nf$ by
    \[
    \Delta_n^{\om,\a}f(z)=\sum_{k\in I_{\om,\a}(n)} a_kz^k,\quad n\in\N\cup\{0\}.
    \]
If $\a=1$, we will simply write $\Delta_n^{\om}$ instead of
$\Delta_n^{\om,1}$.

\begin{theorem}\label{th:dec1}
Let $1<p<\infty$, $0<\a<\infty$ and $\om\in\I\cup\R$ such that
$\int_0^1\om(r)\,dr=1$, and let $f\in\H(\D)$.
\begin{itemize}
\item[\rm(i)] If $0<q<\infty$, then $f\in H(p,q,\om)$ if and only
if
    \begin{equation*}
    \sum_{n=0}^\infty 2^{-n\alpha}\|\Delta^{\om,\a}_n f\|_{H^p}^q<\infty.
    \end{equation*}
Moreover,
    \[
    \|f\|_{H(p,q,\om)}\asymp \bigg(\sum_{n=0}^\infty 2^{-n\a}
    \|\Delta^{\om,\a}_n f\|_{H^p}^q\bigg)^{1/q}.
    \]
\item[\rm(ii)] If $0<\beta<\infty$, then $f\in
H(p,\infty,\ho^\b)$ if and only if
    $$
    \sup_n 2^{-n\a\beta} \| \Delta^{\om,\a}_n f\|_{H^p}<\infty.
    $$
Moreover,
    \[
    \|f\|_{H(p,\infty,\ho^\b)}\asymp \sup_n2^{-n\a\beta}\|
    \Delta^{\om,\a}_n f\|_{H^p}.
    \]
\end{itemize}
\end{theorem}

The method of proof that we use to establish Theorem~\ref{th:dec1}
can be employed to characterize certain functions in $A^p_\om$ in
terms of the coefficients in their Maclaurin series. In fact, we
will see that, whenever $\om\in\R$, a standard lacunary series
$f(z)=\sum_{k=0}^\infty a_k z^{n_k}$, $\frac{n_{k+1}}{n_k}\ge
c>1$, belongs to $A^p_\om$ if and only if
    $$
    \sum_{k=0}^\infty
    |a_k|^q\int_0^1r^{2n_k+1}\om(r)\,dr<\infty.
    $$
The same is not true in general if $\om$ is rapidly increasing.
However, the assertion is valid for $\om\in\I$ if the Maclaurin
series expansion of $f$ has sufficiently large gaps depending
on~$\om$. To give the precise statement, let $\om$ be a radial
weight. We say that $f\in\H(\D)$ is an {\em{$\om$-lacunary
series}} in $\D$ if its Maclaurin series $\sum_{k=0}^\infty
a_kz^{n_k}$ satisfies
    \begin{equation*}
    \frac{\widehat{\om}\left(1-\frac{1}{n_k}\right)}{\widehat{\om}\left(1-\frac{1}{n_{k+1}}\right)}
    =\frac{\int_{1-\frac{1}{n_k}}^1\om(r)\,dr}{\int_{1-\frac{1}{n_{k+1}}}^1\om(r)\,dr}\ge\lambda>1
    ,\quad k\in\N\cup\{0\}.
    \end{equation*}
This is a natural generalization of the classical concept of power
series with Hadamard gaps, in the sense that, for $\om\in\R$, the
class of $\om$-lacunary series is nothing else but the set of
Hadamard gap series.

\begin{theorem}\label{lacunary1}
Let $0<q<\infty$, $0<p\le\infty$ and $\om\in\I\cup\R$ such that
$\int_{0}^1\om(r)\,dr=1$, and let $f$ be an $\om$-lacunary series
in $\D$. Then the following conditions are equivalent:
    \begin{itemize}
    \item[\rm(i)] $f\in H(p,q,\om)$; \item[\rm(ii)]
    $\displaystyle\sum_{k=0}^\infty
    |a_k|^q\int_0^1r^{2n_k+1}\om(r)\,dr<\infty$.
    \end{itemize}
Moreover,
    \begin{equation*}\begin{split}
  &  \|f\|^q_{H(p,q,\om)}\asymp \sum_{k=0}^\infty
    |a_k|^q\int_0^1r^{2n_k+1}\om(r)\,dr.
    \end{split}\end{equation*}
\end{theorem}

The remaining part of the paper is organized as follows. In
Section~\ref{decomposition} we state and prove some preliminary
results on weights and technical results on series with positive
coefficients, and prove Theorems~\ref{th:dec1} and
\ref{lacunary1}. Theorem~\ref{th:gorro} will be proved in
Section~\ref{sublinear}. In Section~\ref{hadamard} we will deal
with technical background on Hadamard products which will be used
in the proof of Theorem~\ref{mainth}, that is given in
Section~\ref{main}. Section~\ref{compact} is devoted to proving
the expected results on the compactness of $\hg:A^p_\om\to
A^q_\om$. Finally, in Section~\ref{further} we will offer natural
alternative descriptions of the spaces appearing in the statement
of Theorem~\ref{mainth} and analyze the Muckenhoupt type condition
\eqref{Mpcondition} in detail. In particular, we will see that
\eqref{Mpcondition} is closely related to the value of $\lim_{r\to
1^-}\p_\om(r)/(1-r)$, if it exists.

\section{Decomposition theorems}\label{decomposition}

This section is instrumental for the rest of the paper. Here we
will discuss basic properties of the radial weights considered and
$L^p_\om$-behavior of power series with positive coefficients, and
then prove Theorem~\ref{th:dec1} and other related decomposition
theorems. We will also prove Theorem~\ref{lacunary1} and further
discuss the $\om$-lacunary series.

\subsection{Preliminaries on weights}

We begin with collecting some necessary definitions and results on
weights in $\I\cup\R$. The \emph{Carleson square} $S(I)$
associated with an interval $I\subset\T$ is the set
$S(I)=\{re^{it}\in\D:\,e^{it}\in I,\, 1-|I|\le r<1\}$, where $|E|$
denotes the Lebesgue measure of the set $E\subset\T$. For
$1<p<\infty$, the letter $p'$ will denote its conjugate, that is,
the number for which $\frac1p+\frac{1}{p'}=1$.

Let $1<p_0<\infty$ and $\eta>-1$. A weight $u$ (not necessarily
radial) satisfies the \emph{Bekoll\'e-Bonami
$B_{p_0}(\eta)$-condition}, denoted by $u\in B_{p_0}(\eta)$, if
there exists a constant $C=C(p_0,\eta,\omega)>0$ such that
    \begin{equation*}
    \begin{split}
    \left(\int_{S(I)}u(z)(1-|z|)^{\eta}\,dA(z)\right)
    \left(\int_{S(I)}u(z)^{\frac{-p'_0}{p_0}}(1-|z|)^{\eta}\,dA(z)\right)^{\frac{p_0}{p'_0}}
    \le C|I|^{(2+\eta)p_0}
    \end{split}
    \end{equation*}
for every interval $I\subset \T$. For the proof of the next
result, see \cite[Lemmas~1.2-1.4]{PelRat}.

\begin{letterlemma}\label{le:condinte}
\begin{itemize}
\item[\rm(i)] Let $\om\in\R$. Then there exist constants
$\a=\a(\om)>0$ and $\b=\b(\omega)\ge\a$ such that
    \begin{equation*}
    \left(\frac{1-r}{1-t}\right)^\a\widehat{\om}(t)\le\widehat{\om}(r)\le
    \left(\frac{1-r}{1-t}\right)^\b\widehat{\om}(t),\quad 0\le r\le
    t<1.
    \end{equation*}
\item[\rm(ii)] Let $\om\in\I$. Then for each $\b>0$ there exists a
constant $C=C(\b,\omega)>0$ such that
    \begin{equation*}
    \widehat{\om}(r)\le C
    \left(\frac{1-r}{1-t}\right)^\b\widehat{\om}(t),\quad 0\le r\le
    t<1.
    \end{equation*}
\item[{\rm(iii)}] For each radial weight $\om$ and $0<\a<1$,
$\widetilde{\om}(r)=\widehat{\om}(r)^{-\a}\om(r)$ is also a weight
and $\psi_{\widetilde{\om}}(r)=\frac1{1-\a}\psi_{{\om}}(r)$ for
all $0<r<1$. \item[{\rm(iv)}] If $\om\in\I\cup\R$, then
    \begin{equation*}
    \int_0^1
    s^{x}\om(s)\,ds\asymp\widehat{\om}\left(1-\frac{1}{x}\right),\quad
    x\in[1,\infty).
    \end{equation*}
\item[{\rm(v)}]      If $\om\in\R$, then for each $p_0>1$ there exists
$\eta_0=\eta(p_0,\omega)>-1$ such that for all $\eta\ge \eta_0$,  $\frac{\om(z)}{(1-|z|)^\eta}$
belongs to $B_{p_0}(\eta)$.\index{Bekoll\'e-Bonami
weight}\index{$B_{p_0}(\eta)$}
    \end{itemize}
\end{letterlemma}

The next lemma is a restatement of \cite[Lemma~2.3]{PelRat}.

\begin{letterlemma}\label{Lemma:Zhu-type}
\begin{itemize}
\item[\rm(i)] If $\om\in\R$, then there exists $\g_0=\g_0(\om)$
such that
    \begin{equation*}
    \int_\D\frac{\om(z)}{|1-\overline{a}z|^{\g+1}}\,dA(z)\asymp\frac{\om\left(S(a)\right)}{(1-|a|)^{\g+1}}\asymp\frac{\om(a)}{(1-|a|)^{\g-1}},\quad
    a\in\D,
    \end{equation*}
for all $\g>\g_0$. \item[\rm(ii)] If $\om\in\I$, then
    \begin{equation*}
    \int_\D\frac{\om(z)}{|1-\overline{a}z|^{\g+1}}\,dA(z)\asymp\frac{\om\left(S(a)\right)}{(1-|a|)^{\g+1}},\quad
    a\in\D,
    \end{equation*}
for all $\g>0$.
\end{itemize}
\end{letterlemma}

\begin{lemma}\label{le:Mncomparable}
Let $\om\in\R$ such that $\int_0^1 \om(r)\,dr=1$, and let
$\{r_n\}_{n=0}^\infty$ be the sequence defined by \eqref{rn} with
$\alpha=1$. Then there exist constants
$\g_2=\gamma_2(\om)>\gamma_1=\gamma_1(\om)>0$ such that
    \begin{equation*}
    2^{\gamma_1}M_n\le M_{n+1}\le 2^{\gamma_2}M_n,\quad r_n\ge \max\left\{ \frac{1}{2^{\gamma_1}}, \frac{1}{2}
\right\}.
    \end{equation*}
\end{lemma}

\begin{proof}
Using  Lemma~\ref{le:condinte}(i) and \eqref{rn}, we obtain
    \begin{equation*}
    \begin{split}
    \frac{M_{n+1}}{M_n}
    \ge\frac{r_{n+1}(1-r_n)}{1-r_{n+1}}
    \ge2^{-\gamma_1}\left(\frac{\widehat{\om}(r_n)}{\widehat{\om}(r_{n+1})}\right)^{1/\beta}
    =2^{\frac{1}{\beta}-\gamma_1},
    \end{split}
    \end{equation*}
where $\b=\b(\om)$ is from Lemma~\ref{le:condinte}(i). The left
hand inequality of the assertion follows by choosing
$\gamma_1=\frac{1}{2\beta}$. The right hand inequality can be
proved in an analogous manner.
\end{proof}

Several useful reformulations of the Muckenhoupt type condition
\eqref{eq:h1} are gathered to the following lemma.

\begin{lemma}\label{Lemma:u_p}
Let $1<p<\infty$ and let $\om$ be a radial weight, and denote
$u_p(r)=(\widehat{\om}(r)(1-r))^{-\frac1p}$. Then the following
conditions are equivalent:
\begin{itemize}
\item[\rm(i)] $\om\in\Mp$; \item[\rm(ii)]
$\widehat{\om}^{-\frac1{p-1}}\in\R$; \item[\rm(iii)] $u_p\in\R$;
\item[\rm(iv)]$\displaystyle
\frac{(1-r)^{p}}{\widehat{\om}(r)}\int_0^r\frac{\widehat{\om}(t)}{(1-t)^p}\,dt\asymp1-r,\quad
0\le r<1$.
\end{itemize}
\end{lemma}

\begin{proof}
(i)$\Leftrightarrow$(ii). Observe first that
    \begin{equation*}
    \begin{split}
    &\left(\int_r^1\widehat{\om}(t)^{-\frac1{p-1}}dt\right)^{p-1}
    \int_0^r\frac{\widehat{\om}(t)}{(1-t)^p}\,dt\\
    &=\left(\frac{\widehat{\om}(r)^\frac1{p-1}\int_r^1\widehat{\om}(t)^{-\frac1{p-1}}dt}{1-r}\right)^{p-1}
    \cdot\frac{\frac{(1-r)^{p}}{\widehat{\om}(r)}\int_0^r\frac{\widehat{\om}(t)}{(1-t)^p}\,dt}{1-r}\\
    &\ge1^{p-1}\cdot\left(\frac{1-(1-r)^{p-1}}{p-1}\right),\quad0\le r<1,
    \end{split}
    \end{equation*}
and hence $\om\in\Mp$ if and only if (ii) and (iv) are satisfied.
Therefore, to see that (i) and (ii) are equivalent, it suffices to
show that (ii) implies (iv). To prove this, note that
    \begin{equation}\label{eq:mp1}
    \widehat{\om}(r)\asymp\frac{(1-r)^{p-1}}{\left(\int_r^1\widehat{\om}(t)^{-\frac1{p-1}}dt\right)^{p-1}},\quad
    0\le r<1,
    \end{equation}
whenever $\widehat{\om}^{-\frac1{p-1}}\in\R$. Therefore, under the
assumption (ii), the condition (iv) is equivalent to
    \begin{equation}\label{7.}
    \left(\int_r^1\widehat{\om}(s)^{-\frac1{p-1}}ds\right)^{p-1}
    \int_0^r\frac{dt}{(1-t)\left(\int_t^1\widehat{\om}(s)^{-\frac1{p-1}}ds\right)^{p-1}}\asymp1.
    \end{equation}
But since $\widehat{\om}^{-\frac1{p-1}}\in\R$,
Lemma~\ref{le:condinte}(i) shows that there exist $\a=\a(p,\om)>0$
and $\b=\b(p,\om)>0$ such that
    \begin{equation}\label{eq:mp2}
     \left(\frac{1-r}{1-t}\right)^\b\le\frac{\int_r^1\widehat{\om}(s)^{-\frac1{p-1}}ds}{\int_t^1\widehat{\om}(s)^{-\frac1{p-1}}ds}
    \le\left(\frac{1-r}{1-t}\right)^\a,\quad 0\le t\le r<1.
     \end{equation}
Hence the left-hand side of \eqref{7.} is dominated by
    $$
    (1-r)^{\a(p-1)}\int_0^r\frac{dt}{(1-t)^{1+\a(p-1)}}\lesssim1,
    $$
and (i)$\Leftrightarrow$(ii) follows. Note that the beginning of
this part of the proof also establishes the implication
(i)$\Rightarrow$(iv).

(ii)$\Leftrightarrow$(iii). If
$\widehat{\om}^{-\frac1{p-1}}\in\R$, then \eqref{eq:mp1} and
\eqref{eq:mp2} yield
    \begin{equation*}
    \begin{split}
    \frac{1}{u_p(r)}\int_r^1u_p(t)\,dt\asymp
    (1-r)\int_r^1\left(\frac{\int_t^1\widehat{\om}(s)^{-\frac1{p-1}}ds}{\int_r^1\widehat{\om}(s)^{-\frac1{p-1}}ds}\right)^\frac{p-1}{p}\frac{dt}{1-t}
    \asymp 1-r,
    \end{split}
    \end{equation*}
that is, $u_p\in\R$. The opposite implication
(iii)$\Rightarrow$(ii) can be proved in a similar manner.

(iv)$\Rightarrow$(iii). A calculation based on the assumption (iv)
shows that $F(r)=(1-r)^{\frac{1}{K}}\int_0^r
\frac{\ho(t)}{(1-t)^p}\,dt$ is increasing on $[0,1)$ for $K>0$
large enough. By using this and (iv) we deduce
    \begin{equation*}
    \begin{split}
    1-r\le\frac{\int_r^1 u_p(s)\,ds}{u_p(r)}&\asymp (1-r)\int_r^1
    \left(\frac{\int_0^r \frac{\ho(t)}{(1-t)^p}\,dt}{\int_0^s
    \frac{\ho(t)}{(1-t)^p}\,dt}\right)^{\frac{1}{p}}\,\frac{ds}{1-s}\\
    &\le(1-r)^{1-\frac{1}{Kp}}\int_r^1
    \frac{ds}{(1-s)^{1-\frac{1}{Kp}}}\asymp(1-r),
    \end{split}
    \end{equation*}
and thus $u_p\in\R$.

Since (i)$\Rightarrow$(iv) by the first part of the proof, the
lemma is now proved.
\end{proof}

\subsection{$L^p_\om$ behavior of power series with positive
coefficients}\label{SubSecL^p-behavior}

We begin with a technical but useful result. Recall that a
function $h$ is called essentially decreasing if there exists a
positive constant $C=C(h)$ such that $h(x)\le C h(y)$ whenever
$y\le x$. Essentially increasing functions are defined in an
analogous manner.

\begin{lemma}\label{le:serie}
Let $\om\in\I\cup\R$ such that $\int_0^1 \om(r)\,dr=1$. For each
$\a>0$ and $n\in\N\cup\{0\}$, let $r_n=r_n(\om,\a)\in[0,1)$ be
defined by \eqref{rn}. Then the following assertions hold:
\begin{enumerate}
\item[\rm(i)] For each $\gamma>0$, there exists
$C=C(\a,\gamma,\om)>0$ such that
    \begin{equation}\label{serie}
    \eta_\gamma(r)=\sum_{n=0}^\infty
    2^{n\gamma}r^{M_n}\le
    C\,\widehat{\om}(r)^{-\frac{\gamma}{\a}},\quad
    0\le r<1.
    \end{equation}
    \item[\rm(ii)] For each $0<\b<1$, there exists
$C=C(\a,\b,\om)>0$ such that
    \begin{equation}\label{momentos}
    2^{-n\a\b}\int_{0}^1
    \frac{r^{M_n}\om(r)}{\widehat{\om}(r)^{\b}}\,dr\le C\int_{0}^1
    r^{M_n}\om(r)\,dr.
    \end{equation}
    \item[\rm(iii)] If $\alpha=1$ in \eqref{rn}, $1<p<\infty$, $p\eta<1$ and $\om\in\R\cap\Mp$, then there exists $C=C(\eta,p,\om)>0$
such that
    \begin{equation}\label{serie3}
   \sum_{n=0}^\infty
   M_n^{1-\frac{1}{p}}
    2^{-n\eta}r^{M_n}\le
    C\frac{\widehat{\om}(r)^\eta}{(1-r)^{1-\frac{1}{p}}},\quad
    0\le r<1.
    \end{equation}
\end{enumerate}
\end{lemma}

\begin{proof}
(i). We will begin with proving \eqref{serie} for $r=r_N$, where
$N\in\N$. To do this, note first that
    \begin{equation}\label{serie2}
    \begin{split}
    \sum_{n=0}^N2^{n\gamma}r_N^{M_n}
    \le \frac{2^\gamma}{2^\gamma-1}
    \,\widehat{\om}(r_N)^{-\frac{\gamma}{\a}}
    \end{split}
    \end{equation}
by \eqref{rn}. To deal with the remainder of the sum, we apply
Lemma~\ref{le:condinte}(i)(ii) and \eqref{rn} to find
$\b=\beta(\om)>0$ and $C=C(\b,\om)>0$ such that
    $$
    \frac{1-r_n}{1-r_{n+j}}\ge C\left(\frac{\widehat{\om}(r_n)}{\widehat{\om}(r_{n+j})}\right)^{1/\b}
    = C2^{\frac{j\a}{\b}},\quad n,j\in\N\cup\{0\}.
    $$
This, the inequality $\log\frac1x\ge1-x$, $0<x\le1$, and
\eqref{rn} give
    \begin{equation*}
    \begin{split}
    \sum_{n=N+1}^\infty 2^{n\gamma}r_N^{M_n}
    &\le2^{N\g}\sum_{j=1}^\infty 2^{j\g}e^{-r_{N+j}\frac{1-r_N}{1-r_{N+j}}}
    \le2^{N\g}\sum_{j=1}^\infty 2^{j\g}e^{-r_2C2^{\frac{j\a}{\b}}}\\
    &=C(\b,\a,\gamma,\om)\,\widehat{\om}(r_N)^{-\frac{\gamma}{\a}}.
    \end{split}
    \end{equation*}
Since $\b=\b(\om)$, this together with \eqref{serie2} gives
\eqref{serie} for $r=r_N$, where $N\in\N$. Now, using standard arguments, it implies \eqref{serie} for any $r\in(0,1)$.

(ii). Clearly,
    \begin{equation}\label{60}
    \begin{split}
    2^{-n\a\b}\int_{0}^{r_n}
    r^{M_n}\widetilde{\om}(r)\,dr&\le
    \frac{2^{-n\a\b}}{\widehat{\om}(r_n)^{\b}} \int_{0}^{r_n}
    r^{M_n}\om(r)\,dr\le \int_{0}^1 r^{M_n}\om(r)\,dr.
    \end{split}
    \end{equation}
Moreover, Lemma~\ref{le:condinte}(iii) yields
    \begin{equation}\label{61}
    \begin{split}
    2^{-n\a\b}\int_{r_n}^1
    r^{M_n}\widetilde{\om}(r)\,dr
    &\le2^{-n\a\b}\widetilde{\om}(r_n)\psi_{\widetilde{\om}}(r_n)
    =\frac{2^{-n\a\b}}{1-\b}\widetilde{\om}(r_n)\psi_{\om}(r_n)\\
    &=\frac{1}{1-\b}\int_{r_n}^1\om(r)\,dr\le C(\b,\a,\om)\int_{r_n}^1 r^{M_n}\om(r)\,dr.
    \end{split}
    \end{equation}
By combining \eqref{60} and \eqref{61} we obtain (ii).

(iii). The proof is similar to that of (i). We will begin with
proving \eqref{serie3} for $r=r_N$, where $N\in\N$. Since $\om\in
\Mp$, Lemma~\ref{Lemma:u_p} yields
$\widehat{\om}^{-\frac{1}{p-1}}\in \R$, that is,
    $$
    \left(\int_r^1\widehat{\om}^{-\frac{1}{p-1}}(s)\,ds\right)^{\frac{p-1}{p}}\asymp (1-r)^{\frac{p-1}{p}}\widehat{\om}(r)^{-\frac1p},
    $$
so taking $r=r_n$ and bearing in mind Lemma~\ref{le:Mncomparable}
we deduce that the sequence
$\left\{\frac{2^{\frac{n}{p}}}{M_n^{\frac{p-1}{p}}}\right\}$ is
essentially decreasing. Therefore
    \begin{equation}
    \begin{split}\label{eq:qp2}
    \sum_{n=0}^N
    M_n^{1-\frac1p}2^{-n\eta}r_N^{M_n}
    &\lesssim  M_N^{1-\frac1p}2^{\frac{-N}{p}}\sum_{n=0}^N 2^{n\left(\frac{1}{p}-\eta\right)}
    \asymp\frac{\widehat{\om}(r_N)^{\eta}}{(1-r_N)^{1-\frac1p}}.
   \end{split}
   \end{equation}
Moreover, bearing in mind Lemma~\ref{le:Mncomparable}, the
inequality $\log\frac1x\ge1-x$, $0<x\le1$, and the boundedness of
the function $x^{s}e^{-tx}$, $s,t>0$, on $[0,\infty)$, we obtain
    \begin{equation*}
    \begin{split}
    \sum_{n=N+1}^\infty M_n^{1-\frac1p}2^{-n\eta}r_N^{M_n}
    &\lesssim 2^{-N\eta}\sum_{j=0}^\infty M_{j+N+1}^{1-\frac1p}2^{-j\eta}e^{-C\frac{M_{j+N+1}}{M_N}}\\
    &\lesssim  2^{-N\eta}  M_{N}^{1-\frac1p}\asymp  \frac{\widehat{\om}(r_N)^{\eta}}{(1-r_N)^{1-\frac1p}},
    \end{split}
    \end{equation*}
which together with \eqref{eq:qp2} gives (iii) for $r=r_N$.
Finally, by using Lemma~\ref{le:Mncomparable}, \eqref{rn} and the
fact that $(1-r)^{\frac{p-1}{p}}\widehat{\om}(r)^{-\frac1p}$ is
essentially decreasing, we obtain (iii) for any $r\in (0,1)$.
\end{proof}

We now present a result on power series with positive
coefficients. This result will play a crucial role in the proof of
Theorem~\ref{th:dec1}.

\begin{proposition}\label{pr:DecompApw}
Let $0<p,\alpha<\infty$ and $\om\in\I\cup\R$ such that $\int_0^1
\om(r)\,dr=1$. Let $f(r)=\sum_{k=0}^\infty a_k r^k$, where $a_k\ge
0$ for all $k\in\N\cup\{0\}$, and denote $t_n=\sum_{k\in
I_{\om,\a}(n)}a_k$. Then there exists a constant $C=C(p,\a,\om)>0$ such that
    \begin{equation}\label{j13}
    \frac{1}{C}\sum_{n=0}^\infty 2^{-n\a}t_n^p\le \int_{0}^1
    f(r)^p\om(r)\,dr\le C \sum_{n=0}^\infty 2^{-n\a}t_n^p.
    \end{equation}
\end{proposition}

\begin{proof}
We will use ideas from the proof of \cite[Theorem~6]{MatelPav}.
The definition \eqref{rn} yields
    \begin{equation*}
    \begin{split}
    \int_{0}^1f(r)^p\om(r)\,dr &\ge
    \sum_{n=0}^\infty\int_{r_{n+1}}^{r_{n+2}}\left(\sum_{k=0}^\infty t_k
    r^{M_{k+1}}\right)^p\om(r)\,dr\\
    & \ge \sum_{n=0}^\infty\left(\sum_{k=0}^n t_k r_{n+1}^{M_{k+1}}\right)^p\int_{r_{n+1}}^{r_{n+2}}\om(r)\,dr\\
    & \ge \left(1-\frac1{2^\a}\right)\sum_{n=0}^\infty t^p_n r_{n+1}^{pM_{n+1}}2^{(-n-1)\a}\ge C \sum_{n=0}^\infty t^p_n2^{-n\a},
    \end{split}
    \end{equation*}
where $C=C(p,\a,\om)>0$ is a constant. This gives the first
inequality in~\eqref{j13}.

To prove the second inequality in~\eqref{j13}, let first $p>1$ and
take $0<\g<\frac{\a}{p-1}$. Then H\"older's inequality gives
    \begin{equation}\label{10.}
    \begin{split}
    f(r)^p\le \left(\sum_{n=0}^\infty t_n r^{M_{n}}\right)^p\le
    \eta_{\g}(r)^{p-1}\sum_{n=0}^\infty 2^{-n\gamma(p-1)}t_n^p
    r^{M_{n}}.
    \end{split}
    \end{equation}
Therefore, by \eqref{serie} and \eqref{momentos} in
Lemma~\ref{le:serie} and  Lemma~\ref{le:condinte}(iv) there exist
constants $C_1=C_1(\a,\g,p,\om)>0$, $C_2=C_2(\a,\g,p,\om)>0$ and
$C_3=C_3(\a,\g,p,\om)>0$ such that
    \begin{equation*}
    \begin{split}
    \int_{0}^1f(r)^p\om(r)\,dr &\le \sum_{n=0}^\infty2^{-n\gamma(p-1)}t_n^p\int_{0}^1
    r^{M_{n}}\eta_{\g}(r)^{p-1}\om(r)\,dr\\
    &\le C_1\sum_{n=0}^\infty  2^{-n\gamma(p-1)}t_n^p\int_{0}^1\frac{r^{M_{n}}\om(r)}{\widehat{\om}(r)^{\frac{\gamma(p-1)}{\a}}}\,dr\\
    &\le C_2\sum_{n=0}^\infty t^p_n\int_{0}^1 r^{M_{n}}\om(r)\,dr\\
    &\le C_3\sum_{n=0}^\infty t^p_n\,\widehat{\om}(r_n)\,dr=C_3\sum_{n=0}^\infty
    t^p_n2^{-n\alpha}.
    \end{split}
    \end{equation*}
Since $\g=\g(\a,p)$, this gives the assertion for $1<p<\infty$.
The proof of the case $0<p\le1$ is similar but easier.
\end{proof}

\subsection{Decomposition theorems}

In this section we will prove Theorem~\ref{th:dec1} and  related
results as well as discuss their consequences. For
$g(z)=\sum_{k=0}^\infty b_k z^k\in\H(\D)$ and
$n_1,n_2\in\N\cup\{0\}$, we set
    $$
    S_{n_1,n_2}g(z)=\sum_{k=n_1}^{n_2-1}b_kz^k,\quad n_1<n_2.
    $$
We will use repeatedly the following auxiliary result.

\begin{lemma}\label{le:4}
Let $0<p\le \infty$ and $n_1,n_2\in\N$ with $n_1<n_2$. If
$g(z)=\sum_{k=0}^{\infty} c_k z^k\in \H(\D)$, then
\begin{equation*}
\| S_{n_1,n_2}g\|_{H^p}\asymp M_p\left(1-\frac{1}{n_2}, S_{n_1,n_2}g\right).
\end{equation*}
\end{lemma}

Lemma~\ref{le:4} can be proved, for example, by using the
inequality
\begin{equation}\label{Eq-MaPa}
     r^{n_2}\|S_{n_1,n_2}g\|_{H^p}\le M_p(r,S_{n_1,n_2}g)\le r^{n_1}\|S_{n_1,n_2}g\|_{H^p}, \quad
     0<r<1,
    \end{equation}
which follows by \cite[Lemma~3.1]{MatelPavstu}.

\medskip

\noindent\emph{Proof of} Theorem~\ref{th:dec1}. (i). By the
M.~Riesz projection theorem and \eqref{Eq-MaPa},
    \begin{equation*}
    \begin{split}
    \|f\|_{H(p,q,\om)}
    &\gtrsim\sum_{n=0}^\infty \|\Delta^{\om,\a}_n f\|_{H^p}^q\int_{r_{n+1}}^{r_{n+2}}
    r^{q{M_{n+1}}}\om(r)\,dr\\
    &\asymp\sum_{n=0}^\infty \|\Delta^{\om,\a}_n f\|_{H^p}^q\int_{r_{n+1}}^{r_{n+2}}\om(r)\,dr
    \asymp\sum_{n=0}^\infty 2^{-n\a}\|\Delta^{\om,\a}_n f\|_{H^p}^q.
    \end{split}
    \end{equation*}
On the other hand, Minkowski's inequality and \eqref{Eq-MaPa} give
    \begin{equation}\label{1}
    M_p(r,f)\le\sum_{n=0}^\infty M_p(r,\Delta^{\om,\a}_nf)\le\sum_{n=0}^\infty
    r^{M_n}\|\Delta^{\om,\a}_n f\|_{H^p},
    \end{equation}
and hence Proposition~\ref{pr:DecompApw} yields
    \begin{equation*}
    \begin{split}
    \|f\|_{H(p,q,\om)}\le\int_0^1 \left( \sum_{n=0}^\infty
    r^{M_n} \|\Delta^{\om,\a}_n f\|_{H^p}\right)^q\om(r)\,dr
    \asymp\sum_{n=0}^\infty 2^{-n\a}\|\Delta^{\om,\a}_n
    f\|_{H^p}^q.
    \end{split}
    \end{equation*}

(ii). Using again the M.~Riesz projection theorem and
\eqref{Eq-MaPa} we deduce
    \[
    \sup_{0<r<1}M_p(r,f)\,\widehat{\om}(r)^\b\gtrsim
    r_{n+1}^{{M_{n+1}}}\|\Delta^{\om,\a}_nf\|_{H^p}2^{-n\a\beta},\quad
    n\in\N\cup\{0\},
    \]
and hence
    $$
    \|f\|_{H(p,\infty,\widehat{\om}^{\b})}\gtrsim\sup_n2^{-n\a\beta}
    \|\Delta^{\om,\a}_nf\|_{H^p}.
    $$
Conversely, assume that $M=\sup_n 2^{-n\a\beta}\|
\Delta^{\om,\a}_n f\|_{H^p}<\infty$. Then \eqref{1} and
Lemma~\ref{le:serie}(i) yield
    \[
    \begin{aligned}
    M_p(r,f)&\le\sum_{n=0}^\infty r^{M_n}\|\Delta^{\om,\a}_n f\|_{H^p}\le M\sum_{n=0}^\infty
    2^{n\a\beta}r^{M_n} \lesssim M
    \widehat{\om}(r)^{-\beta}.
    \end{aligned}
    \]
This finishes the proof. \hfill$\Box$

\medskip

Now we will present a couple of results which will be strongly
used in the proof of Theorem~\ref{mainth}. We saw in
Theorem~\ref{th:dec1} that the $H(p,q,\om)$-norm of $f\in\H(\D)$
can be written in terms of $H^p$-norms of the polynomials
$\Delta^{\om,\a}_nf$ if $\om\in\I\cup\R$. The next result shows
that the same polynomials work also in the case of $H(p,q,\om_\g)$
whenever $\g\ge0$ and $\om\in\R$. In both,
Corollary~\ref{co:mixtogamma} and Corollary~\ref{Lemma:q>p},
$M_n=E\left(\frac1{1-r_n}\right)$, where $r_n$ is defined by
\eqref{rn} with $\a=1$.

\begin{corollary}\label{co:mixtogamma}
Let $1<q<\infty$, $0<p<\infty$, $0\le\g<\infty$, $\om\in\R$ such
that $\int_0^1\om(r)\,dr=1$ and $g\in\H(\D)$. Then
    \begin{equation*}
   \|g\|^p_{H(q,p,\om_\g)}= \int_0^1M_q^p(r,g)(1-r)^\gamma\om(r)\,dr\asymp\sum_{n=0}^\infty
    2^{-n}\frac{\|\De_n^\om g\|_{H^q}^p}{M_n^\g}.
    \end{equation*}
\end{corollary}

\begin{proof}
The inequality
    \begin{equation*}
    \begin{split}
    \int_0^1M_q^p(r,g)(1-r)^\g\om(r)\,dr
    \gtrsim\sum_{n=0}^\infty\frac{2^{-n}}{M_{n}^\g}\|\De_{n}g\|_{H^q}^p
    \end{split}
    \end{equation*}
follows by the M.~Riesz projection theorem, Lemma~\ref{le:4} and
Lemma~\ref{le:Mncomparable}.

On the other hand, by Lemma \ref{le:condinte}(iii) the weight
$\widetilde{\om}_\b(r)=\frac{\om(r)}{\widehat{\om}(r)^\b}$ is
regular for each $\b\in(0,1)$ and then $(1-r)^\g\widetilde{\om}_\b(r)$ is also regular. Therefore
Lemma~\ref{le:condinte}(iv) yields
    \begin{equation}\label{11.}
    \int_0^1r^n(1-r)^\g\widetilde{\om}_\b(r)\,dr
    \asymp\int_{1-\frac1n}^1r^n(1-r)^\g\widetilde{\om}_\b(r)\,dr
    \le\frac1{n^\g}\int_{1-\frac1n}^1r^n\widetilde{\om}_\b(r)\,dr.
    \end{equation}
By \eqref{1}, \eqref{10.}, Lemma~\ref{le:serie}(i), \eqref{11.}
with $\b=\eta(p-1)\in(0,1)$, \eqref{61} and
Lemma~\ref{le:Mncomparable},
    \begin{equation*}
    \begin{split}
    &\int_0^1M_q^p(r,g)(1-r)^\g\om(r)\,dr
    \le\int_0^1\left(\sum_{n=0}^\infty
    r^{M_n}\|\De^\om_ng\|_{H^q}\right)^p(1-r)^\g\om(r)\,dr\\
    &\quad\le\sum_{n=0}^\infty2^{-n\eta(p-1)}\|\De_n^\om
    g\|_{H^q}^p\int_0^1\frac{\om(r)}{\widehat{\om}(r)^{\eta(p-1)}}(1-r)^\g
    r^{M_n}\,dr\\
    &\quad\lesssim\sum_{n=0}^\infty2^{-n\eta(p-1)}\|\De_n^\om
    g\|_{H^q}^p\frac{1}{M_n^\g}\int_{1-\frac1{M_n}}^1\frac{\om(r)r^{M_n}}{\widehat{\om}(r)^{\eta(p-1)}}\,dr\\
    &\quad\lesssim\sum_{n=0}^\infty\|\De_n^\om
    g\|_{H^q}^p\frac{1}{M_n^\g}\int_{1-\frac1{M_n}}^1r^{M_n}\om(r)\,dr
    \asymp\sum_{n=0}^\infty2^{-n}\|\De_n^\om
    g\|_{H^q}^p\frac{1}{M_n^\g},
    \end{split}
    \end{equation*}
and the proof is complete.
\end{proof}

The second result generalizes a known characterization of the mean
Lipschitz space $\Lambda(p,\a)$, where $1<p<\infty$ and $0<\a \le
1$, see \cite[Theorem~2.1-3.1]{MatelPavstu}. We say that
$g\in\lambda(p,\a,\widehat{\om}^\b)$, if
    $$
    \lim_{r\to1^-}\frac{M_p(r,g')(1-r)^{1-\a} }{\ho(r)^\b}=0.
    $$

\begin{corollary}\label{Lemma:q>p}
Let $1<q,p<\infty$, $\eta\in\left[0,\frac{1}{p}\right)$,
$\om\in\R\cap\Mp$ such that $\int_0^1\om(r)\,dr=1$ and
$g\in\H(\D)$.
\begin{itemize}
\item[\rm(i)]
$g\in\Lambda\left(q,\frac{1}{p},\widehat{\om}^\eta\right)$ if and
only if $\displaystyle \|\Denw g'\|_{H^q}\lesssim
M_n^{1-\frac1p}2^{-n\eta}$ for all $n\in\N$. Moreover,
    $$
    \|g\|_{\Lambda\left(q,\frac{1}{p},\widehat{\om}^\eta\right)}\asymp |g(0)|+\sup_n\frac{\|\Denw g'\|_{H^q}2^{n\eta}}{M_n^{1-\frac1p}}.$$
    \item[\rm(ii)]$g\in\lambda(q,\frac{1}{p},\widehat{\om}^\eta)$ if and only if
    $$
    \displaystyle
    \|\Denw g'\|_{H^q}=\op\left(
    M_n^{1-\frac1p}2^{-n\eta}\right),\quad n\to\infty.
    $$
\end{itemize}
\end{corollary}

Corollary~\ref{Lemma:q>p} can be obtained by following the lines
of the proof of Theorem~\ref{th:dec1}~(ii) together with
Lemma~\ref{le:serie}~(iii). We omit the details.

Finally, we will give simple proofs of several known results as a
by-product of Theorem~\ref{th:dec1}.

\begin{lettercorollary}\label{co:prevdecom}
Let $1<p<\infty$, $0<\a<\infty$ and $f(z)=\sum_{k=0}^\infty a_k
z^k\in\H(\D)$.
    \begin{itemize}
    \item[\rm(i)] If $0<\g<\infty$, then
    $$
    \int_0^1 M_p^q(r,f)(1-r)^{q\g-1}\,dr
    \asymp |a_0|^q+\sum_{n=0}^\infty2^{-nq\g}\left\Vert \sum_{k=2^n}^{2^{n+1}-1}a_k z^k\right\Vert^q_{H^p}.
    $$
    \item[\rm(ii)] If $1/q<\beta<\infty$, then
    \begin{equation*}
    \begin{split}
    &\int_0^1M_p^q(r,f)\left(\log\frac{2}{1-r}\right)^{-q\b}(1-r)^{-1}\,dr\\
    &\asymp\left\Vert \sum_{k=0}^{3}a_kz^k\right\Vert^q_{H^p}+\sum_{n=1}^\infty 2^{-n(q\b-1)}
    \left\Vert\sum_{k=2^{2^n}}^{2^{2^{n+1}}-1}a_k z^k\right\Vert^q_{H^p}.
    \end{split}
    \end{equation*}
    \end{itemize}
\end{lettercorollary}

\begin{proof}
(i) Consider the regular and normalized weight
$\om(r)=q\g(1-r)^{q\g-1}$. Then, by choosing $\a=q\g$ in
\eqref{rn} and Theorem~\ref{th:dec1}, the result follows.

(ii) Take the normalized rapidly increasing weight
    $$
    \om(r)=\frac{q\b-1}{\log2}\frac{1}{(1-r)\left(\log_2\frac1{1-r}\right)^{q\b}},\quad
    q\b>1.
    $$
Then, by choosing $\a=q\beta-1$ in \eqref{rn} and
Theorem~\ref{th:dec1}, the result follows.
\end{proof}

Corollary~\ref{co:prevdecom}(i) is obtained in~\cite{MatelPavstu}
as a consequence of a more general result.
Corollary~\ref{co:prevdecom}(ii) is nothing else but
\cite[Theorem~6.1]{GPP}, the original proof of which is more
involved and uses the Riesz-Thorin interpolation theorem.

\subsection{$\mathbf{\om}$-Lacunary series}

The main purpose of this section is to stress how different is a
weighted Bergman space $A^p_\om$, induced by a rapidly increasing
weight $\om$, from another $A^p_\om$, induced by a regular one.
This will be done by using strongly the results on power series
with positive coefficients obtained in
Section~\ref{SubSecL^p-behavior}. The reader is invited to
see~\cite{PelRat} for more information on this topic.

Recall that, for a given radial weight $\om$, $f\in\H(\D)$ is said
to be an $\om$-lacunary series in $\D$ if its Maclaurin series
$\sum_{k=0}^\infty a_kz^{n_k}$ satisfies
    \begin{equation}\label{wlacunary}
    \frac{\widehat{\om}(1-\frac{1}{n_{k}})}{\widehat{\om}(1-\frac{1}{n_{k+1}})}=\frac{\int_{1-\frac{1}{n_k}}^1\om(r)\,dr}{\int_{1-\frac{1}{n_{k+1}}}^1\om(r)\,dr}\ge\lambda>1
    ,\quad k\in\N\cup\{0\}.
    \end{equation}

We begin with proving an extension of Theorem~\ref{lacunary1} that
describes the $\om$-lacunary series in the mixed norm space
$H(p,q,\om)$ in terms of the coefficients in their Maclaurin
series.

\begin{theorem}\label{lacunary}
Let $0<q,\a<\infty$, $0<p\le\infty$ and $\om\in\I\cup\R$ such that
$\int_{0}^1\om(r)\,dr=1$, and let $f$ be an $\om$-lacunary series
in $\D$. Then the following conditions are equivalent:
    \begin{itemize}
    \item[\rm(i)] $f\in H(p,q,\om)$; \item[\rm(ii)]
    $\displaystyle\sum_{n=0}^\infty 2^{-n\a}\left(\sum_{n_k\in
    I_{\om,\a}(n)}|a_k|^2\right)^{q/2}<\infty$; \item[\rm(iii)]
    $\displaystyle\sum_{n=0}^\infty 2^{-n\a}\left(\sum_{n_k\in
    I_{\om,\a}(n)}|a_k|^q\right)<\infty$;
    \item[\rm(iv)]
    $\displaystyle\sum_{n=0}^\infty 2^{-n\a}\left(\sum_{n_k\in
    I_{\om,\a}(n)}|a_k|\right)^{q}<\infty$;
    \item[\rm(v)]
    $\displaystyle\sum_{k=0}^\infty
    |a_k|^q\int_0^1r^{2n_k+1}\om(r)\,dr<\infty$.
    \end{itemize}
Moreover, each of the sums in {\rm{(ii)-(v)}} is comparable to
$\|f\|^q_{H(p,q,\om)}$.
\end{theorem}

\begin{proof} Let $f$ be an $\om$-lacunary series
in $\D$. First, we observe that the chain of inequalities
    $$
    \frac{1}{1-r_n}\le n_k<n_{k+s}<\frac{1}{1-r_{n+1}}
    $$
is equivalent to
    \begin{equation}\label{4.}
    \begin{split}
    \frac{1}{2^{n\a}}&=\widehat{\omega}\left(r_n\right)
    \ge\widehat{\omega}\left(1-\frac{1}{n_k}\right)>\widehat{\omega}\left(1-\frac{1}{n_{k+s}}\right)
    >\widehat{\omega}\left(r_{n+1}\right)=\frac{1}{2^{(n+1)\a}}
    \end{split}
    \end{equation}
by \eqref{rn}. This together with \eqref{wlacunary} shows that
there are at most $\log_{\lambda}2^\a+2$ integers $n_k$ in each
set $I_{\om,\a}(n)$. Therefore H\"older's inequality and standard estimates
give
    \begin{equation*}
    \begin{split}
    \sum_{n=0}^\infty 2^{-n\a}\left(\sum_{k\in I_{\om,\a}(n)}|a_k|^2\right)^{q/2}
    &\asymp\sum_{n=0}^\infty 2^{-n\a}\left(\sum_{k\in
    I_{\om,\a}(n)}|a_k|^q\right)\\
    &\asymp\sum_{n=0}^\infty 2^{-n\a}\left(\sum_{n_k\in
    I_{\om,\a}(n)}|a_k|\right)^{q},
    \end{split}
    \end{equation*}
and thus (ii)$\Leftrightarrow$(iii)$\Leftrightarrow$(iv).
Moreover, by Lemma~\ref{le:condinte}(i) (ii) (iv),
    $$
    \widehat{\omega}\left(1-\frac{1}{n_k}\right)\asymp\widehat{\omega}\left(1-\frac{1}{2n_k+1}\right)\asymp\int_0^1
    r^{2n_k+1}\om(r)\,dr,
    $$
and it follows by \eqref{4.} that (iii)$\Leftrightarrow$(v).

By the proof of \cite[Lemma~1.2]{PelRat}, there exist
$\b=\beta(\om)>0$ and $N\in\N$ such that
    $$
    \frac{\widehat{\omega}\left(1-\frac{1}{n_k}\right)}{\widehat{\omega}\left(1-\frac{1}{n_{k+1}}\right)}\le \left(\frac{n_{k+1}}{n_k}\right)^\b
    $$
for all $k\ge N$. Therefore $f$ is a standard lacunary series by
\eqref{wlacunary}. In fact, Lemma~\ref{le:condinte}(i) shows that
an $\om$-lacunary series for $\om\in\R$ is just a standard
lacunary series. Consequently, if $0<p<\infty$, Zygmund's theorem
\cite[p.~215]{Zygmund59} gives
    \begin{equation*}
    \|f\|_{H(p,q,\om)}^q\asymp\int_0^{1}\left(\sum_{k=0}^\infty
    |a_k|^2r^{2n_k}\right)^{\frac{q}{2}}\omega(r)r\,dr.
    \end{equation*}
Therefore Proposition~\ref{pr:DecompApw} implies
(i)$\Leftrightarrow$(ii) and
    $$
    \|f\|^q_{H(p,q,\om)}\asymp\sum_{n=0}^\infty 2^{-n\a}\left(\sum_{n_k\in
    I_{\om,\a}(n)}|a_k|^2\right)^{q/2}.
    $$
This completes the proof for $0<p<\infty$.

Finally, if $f\in H(\infty,q,\om)$, then $f\in H(p,q,\om)$ for any
$0<p<\infty$, so by the previous argument (i)$\Rightarrow$(ii).
Reciprocally, assume that (iv) holds. Then, by using
Proposition~\ref{pr:DecompApw}, we deduce
    \begin{equation*}
    \|f\|_{H(\infty,q,\om)}^q\le \int_0^{1}\left(\sum_{k=0}^\infty
    |a_k|r^{n_k}\right)^{q}\omega(r)r\,dr\asymp \sum_{n=0}^\infty 2^{-n\a}\left(\sum_{n_k\in
    I_{\om,\a}(n)}|a_k|\right)^{q}<\infty.
    \end{equation*}
This finishes the proof.
\end{proof}

Theorem~\ref{lacunary} gives an easy way to construct functions in
$A^p_\om$. For example, if $0<p<q<\infty$ and $\om\in\I\cup\R$,
then Theorem \ref{lacunary}, with $\alpha=1$, shows that
    $$
    f(z)=\sum_{n=0}^\infty
    2^{n/q}z^{M_n},\quad M_n=E\left(\frac1{1-r_n}\right),
    $$
where $r_n$ is given by \eqref{rn} with $\a=1$, belongs to
$A^p_\om\setminus A^q_\om$.

It is worth noticing that the equivalence (i)$\Leftrightarrow$(ii)
in Theorem~\ref{lacunary} is valid for standard lacunary series
and $\om\in\I\cup\R$. However, (i)$\Leftrightarrow$(iii) is no
longer true for standard lacunary series if $\om\in\I$ and $q\neq 2$. Namely,
let us consider the rapidly increasing weight
    $$
    v_\beta(r)= (1-r)^{-1}\left(\log\frac{e}{1-r}\right)^{-\beta},\quad\beta>1.
    $$
If (i) and (iii) were equivalent, then the choice $\alpha=\beta
-1$ would imply that a standard lacunary series
$f(z)=\sum_{n=0}^\infty a_nz^{2^n}$ belongs to $A^q_{v_\b}$ if and
only if
    $$
    \sum_{n=0}^\infty 2^{-n(\beta-1)}\sum_{k=2^{2^n}}^{2^{2^{n+1}}-1}|a_k|^q\asymp \sum_{k=4}^\infty |a_k|^q \left(\log k\right)^{-\b+1}<\infty.
    $$
But this is impossible. Namely, if $\beta>2$ and $a_k=k^{-1/p}$,
then we would have $f\in A^q_{v_\b}$ for $q\ge p$, but $f\notin
A^q_{v_\b}$ for $q<p$. A similar reasoning also works for
$1<\b<2$. An analogous argument can be used to show that the
condition (v) does not characterize standard lacunary series in
$A^p_\om$ when $\om\in\I$ and $q\neq 2$.

If $\om\in\I$, then \eqref{wlacunary} says, roughly speaking, that
the smaller the space $A^p_\om$ is, the larger the gaps of an
$\om$-lacunary series are. Namely, the condition (iii) in
Theorem~\ref{lacunary} is equivalent to (i) and (ii) when the
series $\sum a_kz^{n_k}$ has very large gaps depending on $\om$.

The next result offers a description of $\om$-lacunary series in
the mixed norm space $H(p,\infty,\ho^\b)$.

\begin{theorem}\label{th:lacunaryinfinity}
Let $0<\beta<\infty$ and $\om\in\I\cup\R$ such that $\int_{0}^1\om(r)\,dr=1$. Let
$f(z)=\sum_{n=0}^\infty a_kz^{n_k}$ be an $\om$-lacunary series in
$\D$. Then the following assertions are equivalent:
 \begin{itemize}
    \item[\rm(i)]\, $f\in H(\infty,\infty,\ho^\b)$;
    \item[\rm(ii)]\, $f\in H(p,\infty,\ho^\b)$ for some
    $0<p\le\infty$;
    \item[\rm(iii)]\, The coefficients $\{a_k\}$ of the Maclaurin series
    of $f$ satisfy
    \begin{equation}\label{3}
    |a_k|\lesssim\left(\int_{0}^1r^{n_k}\om(s)\,ds\right)^{-\beta},\quad
    k\in\N\cup\{0\}.
    \end{equation}
     \end{itemize}
\end{theorem}

\begin{proof}
The implication (i)$\Rightarrow$(ii) is trivial. Moreover, as each
$\om$-lacunary series is a standard lacunary series, $f\in
H(p,\infty,\ho^\b)$ if and only if $f\in H(2,\infty,\ho^\b)$.
Therefore Cauchy integral formula and Lemma~\ref{le:condinte}(iv)
easily give (ii)$\Rightarrow$(iii). To complete the proof, we will
establish (iii)$\Rightarrow$(i). If we choose $\a=\frac{1}{\b}$ in
\eqref{rn}, then Lemma~\ref{le:serie}(i) gives
    $$
   \sum_{n=1}^\infty2^n |z|^{M_n}\lesssim\widehat{\om}(|z|)^{-\beta},\quad z\in\D,
    $$
so it suffices to prove
    $
    \sum_{k=1}^\infty\frac{r^{n_k}}{\widehat{\om}\left(1-\frac1{n_k}\right)^{\beta}}
    \lesssim\sum_{n=1}^\infty2^nr^{M_n}.
    $
Bearing in mind Lemma~\ref{le:condinte}(i)-(ii) and arguing as in
the proof of Theorem~\ref{lacunary}, we deduce
    \begin{equation*}
    \begin{split}
    \sum_{k=1}^\infty\frac{r^{n_k}}{\widehat{\om}\left(1-\frac1{n_k}\right)^{\beta}}
    &=\sum_{n=1}^\infty\sum_{n_k\in
    I_{\om,\frac{1}{\beta}}(n)}\frac{r^{n_k}}{\widehat{\om}\left(1-\frac1{n_k}\right)^{\beta}}\\
    &\le\sum_{n=1}^\infty r^{M_n}\sum_{n_k\in
    I_{\om,\frac{1}{\beta}}(n)}\frac{1}{\widehat{\om}\left(1-\frac1{M_{n+1}-1}\right)^{\beta}}\\
    &\le(\log_\lambda 2^{1/\beta}+2)\sum_{n=1}^\infty
    \frac{r^{M_n}}{\widehat{\om}\left(r_{n+1}\right)^{\beta}},
    \end{split}
    \end{equation*}
which together with \eqref{rn} finishes the proof.
\end{proof}

Theorem~\ref{th:lacunaryinfinity} generalizes and improves know
results in the existing literature. In particular, by taking the
regular weight $\phi_\g(r)=\log^\g\frac{e}{1-r}$ and choosing
$\om$ such that $\phi_\g([0,1])\cdot\om(r)=\phi_\g(r)$, we deduce
that the lacunary series $f(z)=\sum_{n=0}^\infty a_kz^{n_k}$,
where $\frac{n_{k+1}}{n_k}\ge\lambda>1$, satisfies the Bloch-type
condition
    $$
    M_\infty(r,f')=\og\left(\frac{1}{(1-r)\log^\g\frac{e}{1-r}}\right),\quad
    0<\gamma<\infty,
    $$
if and only if
    \begin{equation*}
    |a_k|=\og\left(\left(\log n_k\right)^{-\g}\right),\quad
    k\in\N.
    \end{equation*}

\section{The role of the sublinear Hilbert operator}\label{sublinear}

The generalized Hilbert operator
    $$
    \mathcal{H}_g(f)(z)=\int_0^1f(t)g'(tz)\,dt
    $$
is well defined whenever
    \begin{equation}\label{wd}
    \int_0^1|f(t)|\,dt<\infty.
    \end{equation}
Further, if $f(z)= \sum_{n=0}^\infty a_nz^n\in \H(\D)$ satisfies
\eqref{wd}, then $\hg(f)$ can be  written in terms of the
coefficients of the Maclaurin series of $f$ and $g$. Namely, if
$g(z)=\sum_{n=0}^\infty b_nz^n\in \H (\D )$, then
    \begin{equation*}
    \begin{split}\label{Hgcoef}
    \mathcal{H}_g(f)(z)&=\sum_{k=0}^{\infty} \left( (k+1)b_{k+1}\int_0^1t^k f(t)\,dt\right)z^k\\
    &=\sum_{k=0}^\infty \left
    ((k+1)b_{k+1}\sum_{n=0}^\infty \frac{a_n}{n+k+1}\right)z^k.
    \end{split}
    \end{equation*}

We begin with noting that condition \eqref{99} implies \eqref{wd}
for any $f\in A^p_\om$. In fact, by using H\"{o}lder's inequality
and \eqref{le:minfty}, we deduce
    \begin{equation*}
    \begin{split}
    \int_0^1|f(t)|\,dt \le \left(\int_0^1
    |f(t)|^p\ho(t)\,dt\right)^{\frac{1}{p}}\left(\int_0^1\ho(t)^{-\frac{p'}{p}}\,dt\right)^{\frac{1}{p'}}
    \lesssim \|f\|^p_{A^p_\om}.
    \end{split}
    \end{equation*}
The standard radial weight $(1-|z|^2)^\alpha$ satisfies \eqref{99}
if and only if $-1<\alpha<p-2$. Moreover, the function
$h(z)=(1-z)^{-1}\left(\log\frac{e}{1-z}\right)^{-1}$ belongs to
$A^p_{p-2}$ for all $1<p<\infty$, but $\int_0^1|h(t)|\,dt=\infty$.
Therefore \eqref{99} is a natural sharp condition for both, the
generalized Hilbert operator $\hg$ and the sublinear Hilbert
operator
    $$
    \hti(f)(z)=\int_0^1\frac{|f(t)|}{1-tz}\,dt
    $$
to be well defined. As mentioned in \eqref{justificationmaximal}, the
operator $\hti$ behaves like a maximal operator with respect to
$\hg$ under appropriate hypotheses on $\om$ and $g$. Consequently, in
view of \eqref{le:minfty}, it is natural to study the boundedness
of $\hti$ on both $L^p_{\ho}$ and $A^p_\om$. This is the main aim
of this section.

\medskip

\noindent\emph{Proof of} Theorem~\ref{th:gorro}.
(i)$\Rightarrow$(iii). This part of the proof uses ideas from
\cite{Muckenhoupt1972}. For $r\in[0,1)$, set
$\phi_r(t)=\ho(t)^{-\frac{1}{p-1}}\chi_{[r,1)}(t)$, so that
$\phi_r\in L^p_{\ho}$ for all $r\in[0,1)$ by~\eqref{99}. Here, as
usual, $\chi_E$ stands for the characteristic function of the set
$E$. Then, bearing in mind \eqref{le:minfty}, we deduce
    \begin{equation*}
    \begin{split}
    \|\h(\phi_r)\|_{L^p_{\ho}}\lesssim \|\h(\phi_r)\|_{A^p_\om} \le
    \|\h\|_{\left(L^p_{\ho},A^p_\om\right)}\|\phi_r\|_{L^p_{\ho}},
    \end{split}
    \end{equation*}
and hence
    \begin{equation}\label{eq:j12}
    \int_{0}^1 \ho(s)\left(\int_r^1\frac{\ho^{-\frac{1}{p-1}}(t)}{1-ts}\,dt\right)^p\,ds \lesssim \int_r^1 \ho(t)^{-\frac{1}{p-1}}\,dt.
    \end{equation}
Since
    \begin{equation*}
    \begin{split}
    \int_{0}^r
    \ho(s)\left(\int_r^1\frac{\ho^{-\frac{1}{p-1}}(t)}{1-ts}\,dt\right)^p\,ds
    \ge \frac{1}{2^p}\left(\int_{0}^r \frac{\ho(s)}{(1-s)^p}\,ds\right)\left(\int_r^1 \ho(t)^{-\frac{1}{p-1}}\,dt\right)^p,
    \end{split}
    \end{equation*}
this together with \eqref{eq:j12} implies $\om\in \Mp$ and
    $$
    \Mp(\om)\lesssim \|\h\|_{\left(L^p_{\ho},A^p_\om\right)}.
    $$
This argument also proves (ii)$\Rightarrow$(iii).

(iii)$\Rightarrow$(i). Since $\om\in\R$ by the assumption, $\om$
is comparable to the differentiable weight
$\frac{\int_r^1\om(s)\,ds}{(1-s)}$, so, by using
\cite[Theorem~1.1]{PavP}, we deduce
    $$
    \|f\|_{A^p_\om}^p\asymp|f(0)|^p+\int_\D|f'(z)|^p(1-|z|)^{p}\,\om(z)\,dA(z),\quad
    f\in\H(\D).
    $$
 Now, for any $\phi\in L^p_{\ho}$,
    $$
    \left(\h(\phi)\right)'(z)=\int_0^1\frac{t\phi(t)}{(1-tz)^{2}}\,dt,
    $$
and so Minkowski's inequality in continuous form yields
    \begin{equation*}
    \begin{split}
    M_p(r,\left(\h(\phi)\right)')&=\left(\frac{1}{2\pi}\int_0^{2\pi}\left|\int_0^1\frac{\phi(t)t}{(1-tre^{i\t})^{2}}\,dt\right|^p\,d\t\right)^\frac1p\\
    &\le\int_0^1\phi(t)\left(\int_0^{2\pi}\frac{d\t}{|1-tre^{i\t}|^{2p}}\right)^\frac1p\,dt
    \asymp\int_0^1\frac{\phi(t)}{(1-tr)^{2-\frac1p}}\,dt,
    \end{split}
    \end{equation*}
and hence
    \begin{equation}
    \begin{split}\label{eq:j13}
    \|\h(\phi)\|^p_{A^p_\om}\lesssim I_1(r)+I_2(r)+|\H(\phi)(0)|^p
    \end{split}
    \end{equation}
where
    $$
    I_1(r)=\int_0^1\left(\int_0^r\frac{\phi(t)}{(1-t)^{2-\frac1p}}\,dt\right)^p(1-r)^{p}\om(r)\,dr
    $$
and
    $$
    I_2(r)=\int_0^1\left(\int_r^1\frac{\phi(t)}{(1-tr)^{2-\frac1p}}\,dt\right)^p(1-r)^{p} \om(r)\,dr.
    $$

We observe that
    \begin{equation}
    \begin{split}\label{eq:j14}
    I_1(r)\lesssim \|\phi\|^p_{L^p_{\ho}}
    \end{split}
    \end{equation}
can be written as
    $$
    \int_0^1\left(\int_0^r \Phi(t)\,dt\right)U^p(r)\,dr\le \int_0^1 \Phi^p(r)V^p(r)\,dr,
    $$
where
    $$
    U^{p}(x)=\left\{
        \begin{array}{cl}
        (1-x)^{p-1}\widehat{\om}(x), &   0\le x<1\\
        0, & x\ge1
        \end{array}\right.,
    $$
    $$
    V^{p}(x)=\left\{
        \begin{array}{cl}
        (1-x)^{2p-1}\widehat{\om}(x), &   0\le x<1\\
        0, & x\ge1
        \end{array}\right.,
    $$
and $\Phi(t)=\frac{\phi(t)}{(1-t)^{2-\frac1p}}$. Since $\ho$ is decreasing,
    \begin{equation*}
    \begin{split}
    &\left(\int_r^1 U^{p}(s)\,ds\right)^{\frac{1}{p}}\left(\int_0^r V^{-p'}(s)\,ds\right)^{\frac{1}{p'}}\\
    &=\left(\int_r^1 (1-s)^{p-1}\ho(s)\,ds\right)^{\frac{1}{p}}
    \left(\int_0^r\frac{1}{(1-s)^{\left(2-\frac{1}{p}\right)p'}\ho^{\frac{p'}{p}}(s)}\,ds\right)^{\frac{1}{p'}}\\
    &\le \ho(r)^{\frac{1}{p}}(1-r)\ho(r)^{-\frac{1}{p}}
    \left(\int_0^r\frac{1}{(1-s)^{\left(2-\frac{1}{p}\right)p'}}\,ds\right)^{\frac{1}{p'}}
    \le C,
    \end{split}
    \end{equation*}
for all $r\in[0,1)$. Now, \cite[Theorem~1]{Muckenhoupt1972} shows
that \eqref{eq:j14} holds. Moreover, since $\om\in\M_p$, by
applying \cite[Theorem~2]{Muckenhoupt1972} with
    $$
    U^{p}(x)=\left\{
        \begin{array}{cl}
        \frac{\widehat{\om}(x)}{(1-x)^p}, &   0\le x<1\\
        0, & x\ge1
        \end{array}\right.,
    $$
and
    $$
    V^{p}(x)=\left\{
        \begin{array}{cl}
        \widehat{\om}(x), &   0\le x<1\\
        0, & x\ge1
        \end{array}\right.,
    $$
we deduce
    \begin{equation*}
    \begin{split}
    I_2(r)\lesssim\int_0^1\left(\int_r^1\phi(t)\,dt\right)^p\frac{\widehat{\om}(r)}{(1-r)^p}\,dr\lesssim \M^p_p(\om)\|\phi\|^p_{L^p_{\ho}},
    \end{split}
    \end{equation*}
which together with \eqref{eq:j13} and \eqref{eq:j14} gives
(iii)$\Rightarrow$(i) and
    $$
    \|\h\|_{\left(L^p_{\ho},A^p_\om\right)}\lesssim \Mp(\om).
    $$
It is clear that the same argument proves
(iii)$\Rightarrow$(ii).\hfill$\Box$

\medskip

It is worth noticing that the implication (iii)$\Rightarrow$(i)
(as well as (iii)$\Rightarrow$(ii)) can also be proved by using
the theory of Bekoll\'e-Bonami weights. We will only give an
outline of this argument. It is strongly based on the following
essentially known result, which follows from
Lemma~\ref{le:condinte}(v) and \cite[Theorem 2.1]{LueInd85}.

\begin{lemma}\label{le:duality}
Let $1<p<\infty$ and $\om\in\R$. Then there exists
$\eta_0=\eta_0(p,\om)>-1$ such that for all $\eta\ge\eta_0$, the
dual of $A^p_\om$ can be identified with
$A^{p'}_{\om^{-\frac{p'}{p}}(1-|z|)^{p'\eta}}$ under the pairing
    \begin{equation}\label{eq:etapairing}
    \langle f,g\rangle_\eta=\int_\D f(z)\overline{g(z)}(1-|z|)^\eta\,dA(z).
    \end{equation}
Reciprocally, the dual of
$A^{p'}_{\om^{-\frac{p'}{p}}(1-|z|)^{p'\eta}}$ can be identified
with $A^p_\om$ under the same pairing.
\end{lemma}

\medskip

\noindent\emph{An alternative proof of}
{\rm(iii)}$\Rightarrow${\rm(i)}. Let $\eta_0=\eta_0(p,\om)>-1$ be
that of Lemma~\ref{le:duality} and fix $\eta\ge\eta_0$. For
simplicity, we write $v_{p'}(z)=
\om(z)^{-\frac{p'}{p}}(1-|z|)^{p'\eta}$. By
Lemma~\ref{le:duality}, the dual of  $A^{p'}_{v_{p'}}$ can be
identified with $A^p_\om$ under the pairing defined by
\eqref{eq:etapairing}. Therefore $\h:L^p_{\ho}\to A^p_\om$ is
bounded if and only if
    \begin{equation*}
    \left|\langle \h(\phi),h \rangle_\eta \right|
    \lesssim\|\phi\|_{L^p_{\ho}}\|h\|_{v_{p'}}, \quad \phi\in L^p_{\ho},\quad h\in
    A^{p'}_{v_{p'}}.
    \end{equation*}
To prove this, let $\phi\in L^p_{\ho}$ and $h\in A^{p'}_{v_{p'}}$. By
Fubini's theorem, the Cauchy integral formula and H\"older's
inequality, we deduce
    \begin{equation}\label{segundaprueba}
    \begin{split}
    \left|\langle \h(\phi), h \rangle_\eta\right|
    &=2\left|\int_0^1\phi(t) \left(\int_0^1\overline{h(r^2t)}r(1-r)^\eta\,dr\right)\,dt\right|
   \lesssim\|\phi\|_{L^p_{\ho}}I(h),
    \end{split}
    \end{equation}
where
    $$
    I(h)=\left(\int_0^1\left(\int_0^1|h(r^2t)|r(1-r)^\eta\,dr\right)^{p'}\ho^{-\frac{p'}{p}}(t)\,dt\right)^{\frac{1}{p'}}.
    $$
A change of variable, the hypotheses $\om\in\Mp$,
\cite[Theorem~1]{Muckenhoupt1972} and \eqref{eq:r0} give
    \begin{equation}
    \begin{split}\label{eq:j5}
    &\int_{\frac{1}{2}}^1\left(\int_0^1|h(r^2t)|r(1-r)^\eta\,dr\right)^{p'}
    \ho^{-\frac{p'}{p}}(t)\,dt\\
    &\lesssim \int_{0}^1\left(\int_0^t M_\infty(s,h)(1-s)^\eta\,ds\right)^{p'}
    \ho^{-\frac{p'}{p}}(t)\,dt\\
    &\lesssim \Mp^{p'}(\om)\int_{0}^1 M^{p'}_\infty(t,h)(1-t)^{p'\eta+1} \om^{-\frac{p'}{p}}(t)\,dt.
    \end{split}
    \end{equation}
Now, it follows from \cite[p.~9 (i)]{PelRat} that
    \begin{equation*}
    \begin{split}
    \int_{t}^1 v_{p'}(s)\,ds
    \ge  \int_{t}^{\frac{1+t}{2}}(1-s)^{p'\eta} \om^{-\frac{p'}{p}}(s)\,ds
    \asymp (1-t)^{p'\eta+1} \om^{-\frac{p'}{p}}(t).
    \end{split}
    \end{equation*}
Consequently, this together with \eqref{eq:j5} and
\eqref{le:minfty} yield
    \begin{equation*}
    \begin{split}
    &\int_{\frac{1}{2}}^1\left(\int_0^1|h(r^2t)|r(1-r)^\eta\,dr\right)^{p'}\ho^{-\frac{p'}{p}}(t)\,dt\\
    &\lesssim\Mp^{p'}(\om)\int_{0}^1 M^{p'}_\infty(t,h)\left(  \int_{t}^1 v_{p'}(s)\,ds \right)\,dt\lesssim  \Mp^{p'}(\om) \|h\|^{p'}_{v_{p'}}.
    \end{split}
    \end{equation*}
By combining this and \eqref{segundaprueba}, the proof is finished.\hfill$\Box$

\section{Background on smooth Hadamard products }\label{hadamard}

If $W(z)=\sum_{k\in J}b_kz^k$ is a polynomial and
$f(z)=\sum_{k=0}^{\infty}a_kz^k\in \H(\D)$, then the Hadamard
product
    $$
    (W\ast f)(z)=\sum_{k\in J}b_ka_kz^k
    $$
is well defined. Further, if $f\in H^1$, then
    $$
    (W\ast f)(e^{it})=\frac{1}{2\pi}\int_0^{2\pi}W(e^{i(t-\theta)})f(e^{i\theta})\,d\theta
    $$
is the usual convolution.

If $\Phi:\mathbb{R}\to\C$ is a $C^\infty$-function such that its
support $\supp(\Phi)$ is a compact subset of $(0,\infty)$, we set
    $$
    A_\Phi=\max_{s\in\mathbb{R}}|\Phi(s)|+\max_{s\in\mathbb{R}}|\Phi''(s)|,
    $$
and we consider the polynomials
    $$
    W_N^\Phi(z)=\sum_{k\in\mathbb
    N}\Phi\left(\frac{k}{N}\right)z^k,\quad N\in\N.
    $$
With this notation we can state the next result on smooth partial
sums.

\begin{lettertheorem}\label{th:cesaro}
Let $\Phi:\mathbb{R}\to\C$ be a $C^\infty$-function such that
$\supp(\Phi)\subset(0, \infty)$ is compact. Then the following
assertions hold:
\begin{itemize}
\item[\rm(i)] There exists a constant $C>0$ such that
    $$
    \left|W_N^\Phi(e^{i\theta})\right|\le C\min\left\{
    N\max_{s\in\mathbb{R}}|\Phi(s)|,
    N^{1-m}|\theta|^{-m}\max_{s\in\mathbb{R}}|\Phi^{(m)}(s)|
    \right\},
    $$
for all $m\in\N\cup\{0\}$, $N\in\N$ and $0<|\theta|<\pi$.
\item[\rm(ii)] There exists a constant $C>0$ such that
    $$
    \left|(W_N^\Phi\ast f)(e^{i\theta})\right|\le CA_\Phi
    M(|f|)(e^{i\theta})
    $$
for all $f\in H^1$. Here $M$ denotes the Hardy-Littlewood
maximal-operator
    $$
    M(|f|)(e^{i\theta})=\sup_{0<h<\pi}\frac{1}{2h}\int_{\theta-h}^{\theta+h}|f(e^{it})|\,dt.
    $$
\item[\rm(iii)] For each $p\in(1,\infty)$ there exists a constant
$C=C(p)>0$ such that
    $$
    \|W_N^\Phi\ast f\|_{H^p}\le C A_\Phi\|f\|_{H^p}
    $$
for all $f\in H^p$. \item[\rm(iv)] For each $p\in(1,\infty)$ and a
radial weight $\om$, there exists a constant $C=C(p,\om)>0$ such
that
    $$
    \|W_N^\Phi\ast f\|_{A^p_\om}\le C A_\Phi\|f\|_{A^p_\om}
    $$
for all $f\in A^p_\om$.
\end{itemize}
\end{lettertheorem}

Theorem~\ref{th:cesaro} follows from the results and proofs in
\cite[p.~111-113]{Pabook}. We will also need the following lemma
whose proof follows from \eqref{Eq-MaPa} and
Lemma~\ref{le:condinte}.

\begin{lemma}\label{le:eqno}
Let $0<p<\infty$, $n_1,n_2\in\N$ with $n_1\le n_2\le Cn_1$, $\om\in\I\cup\R$ and $g\in\H(\D)$.
Then
    \begin{equation*}
    \begin{split}
    \|S_{n_1,n_2}g\|_{A^p_\om}
    &\asymp
    \left(\int_{1-\frac{1}{n_1}}^1\om(s)\,ds\right)^{1/p}\|S_{n_1,n_2}g\|_{H^p}\\
    &\asymp\left(\int_{1-\frac{1}{n_2}}^1\om(s)\,ds\right)^{1/p}\|S_{n_1,n_2}g\|_{H^p}.
    \end{split}
    \end{equation*}
\end{lemma}

The next auxiliary result allows us to prove the maximality of the
sublinear Hilbert operator $\hti$ in the study of the boundedness
of $\hg$ on weighted Bergman spaces. The proof of
Lemma~\ref{le:n14} is analogous to that of
\cite[Lemma~7]{GaGiPeSis} and is therefore omitted.

\begin{lemma}\label{le:n14}
Let $1<p<\infty$, $\om$ be a radial weight satisfying
\eqref{99} and $n_1,n_2\in\N$ with $n_1<n_2$. Let $f\in
A^p_\om$, $g(z)=\sum_{k=0}^{\infty} c_k z^k\in \H(\D)$ and
$h(z)=\sum_{k=0}^{\infty}c_k\left(\int_0^1t^{k}f(t)\,dt\right)z^k$.
Then there exists a constant $C=C(p)>0$ such that
    \begin{equation*}
    \|S_{n_1,n_2}h\|_{H^p} \le C\left(\int_0^1
    t^{\frac{n_1}{4}}|f(t)|\,dt\right)\| S_{n_1,n_2}g\|_{H^p}.
    \end{equation*}
\end{lemma}

The next known result can be proved by summing by parts and using
the M.~Riesz projection theorem~\cite{GaGiPeSis,LaNoPa}.

\begin{letterlemma}\label{le:A}
Let $1<p<\infty$ and
$\lambda=\left\{\lambda_k\right\}_{k=0}^\infty$ be a monotone
sequence of positive numbers. Let $(\lambda
g)(z)=\sum_{k=0}^{\infty} \lambda_kb_k z^k$, where
$g(z)=\sum_{k=0}^\infty b_k z^k$.
\begin{itemize}
\item[\rm(i)] If $\left\{\lambda_k\right\}_{n=0}^\infty$ is
nondecreasing, then there exists a constant $C>0$ such that
    $$
    C^{-1}\lambda_{n_1}\|S_{n_1,n_2} g\|_{H^p}\le\|S_{n_1,n_2}\lambda g\|_{H^p}
    \le C\lambda_{n_2}\|S_{n_1,n_2}g\|_{H^p}.
    $$
\item[\rm(ii)] If $\left\{\lambda_n\right\}_{n=0}^\infty$ is
nonincreasing, then there exists a constant $C>0$ such that
    $$
    C^{-1}\lambda_{n_2}\|S_{n_1,n_2} g\|_{H^p}
    \le\|S_{n_1,n_2} \lambda g\|_{H^p}
    \le C\lambda_{n_1}\|S_{n_1,n_2} g\|_{H^p}.
    $$
\end{itemize}
 \end{letterlemma}

\section{Proof of Theorem~\ref{mainth}. }\label{main}

We may assume without loss of generality that $\int_0^1
\om(r)\,dr=1$. Throughout the proof $\{r_n\}_{n=0}^\infty$ is the sequence
defined by \eqref{rn} with $\alpha=1$.

\subsection{Sufficiency}

Theorem~\ref{th:dec1}, with $\alpha=1$, shows that
    \begin{equation}
    \begin{split}\label{eq:sb1}
    \left\|\hg(f)\right\|^q_{A^q_\om}
    \asymp&\sum_{n=0}^\infty2^{-n}\left\|\Delta^\om_n\hg(f)\right\|^q_{H^q}
    \end{split}
    \end{equation}
for all $f\in\H(\D)$. Now Lemma~\ref{le:n14},
H\"older's inequality and \eqref{le:minfty} yield
    \begin{equation}
    \begin{split}\label{eq:sb2}
    \left\|\Delta^\om_0\hg(f)\right\|_{H^q}
    &\lesssim|g'(0)|\int_0^1|f(t)|\,dt+\left\|S_{M_1,1}\hg(f)\right\|_{H^q}\\
    &\lesssim |g'(0)|\int_0^1|f(t)|\,dt+\left\|S_{M_1,1}g'\right\|_{H^q}\int_0^1 t^{1/4}|f(t)|\,dt\\
    &\lesssim\left(|g'(0)|+\left\|S_{M_1,1}g'\right\|_{H^q}\right)\int_0^1|f(t)|\,dt\\
    &\lesssim \left(\int_0^1 M_{\infty}^p(t,f)\,\widehat{\om}(t)\,dt\right)^{1/p}\lesssim \left\|f\right\|_{A^p_\om},
    \end{split}
    \end{equation}
where the constants of comparison depend on $p$, $q$, $\om$ and
$g$.

Let first $1<p\le q<\infty$ and assume that
$g\in\Lambda\left(q,\frac{1}{p},\ho^{\frac{1}{p}-\frac{1}{q}}\right)$,
that is,
    $$
    M_q(r,g')\le \|g-g(0)\|_{\Lambda\left(q,\frac{1}{p},\ho^{\frac{1}{p}-\frac{1}{q}}\right)}\frac{\widehat{\om}(r)^{\frac1p-\frac1q}}{(1-r)^{1-\frac1p}},\quad
    0\le r<1.
    $$
Lemma~\ref{le:n14}, Lemma~\ref{le:4}, the M.~Riesz projection
theorem and the assumption give
    \begin{equation}
    \begin{split}\label{eq:sb3}
    \left\|\Delta^\om_n\hg(f)\right\|^q_{H^q}
    &\lesssim\left(\int_0^1 t^{\frac{M_n}{4}}|f(t)|\,dt\right)^q\left\|\Delta^\om_n g'\right\|^q_{H^q}\\
    &\lesssim\left(\int_0^1 t^{\frac{M_n}{4}}|f(t)|\,dt\right)^qM^q_q\left(1-\frac{1}{M_{n+1}},g'\right)\\
    &\lesssim
    \left(\int_0^1 t^{\frac{M_n}{4}}|f(t)|\,dt\right)^q
    \widehat{\om}\left(1-\frac1{M_{n+1}}\right)^{q(\frac1p-\frac1q)}M^{q(1-\frac1p)}_{n+1},
    \end{split}
    \end{equation}
where in the last inequality the constant of comparison depend on
$\|g-g(0)\|^q_{\Lambda\left(q,\frac{1}{p},\ho^{\frac{1}{p}-\frac{1}{q}}\right)}$.

Let $f\in A^p_\om$. Then \eqref{le:minfty} yields
$M_\infty(r,f)\lesssim
u_p(r)=\left((1-r)\ho(r)\right)^{-\frac{1}{p}}$. This together the
fact that $u_p\in\R$ by Lemma~\ref{Lemma:u_p}, and
Lemma~\ref{le:condinte}(iv) yields
    \begin{equation*}
    \int_0^1 t^\frac{M_n}{4}|f(t)|\,dt\lesssim \int_0^1 t^\frac{M_n}{4}u_p(t)\,dt
    \asymp\widehat{u}_p\left(1-\frac1{M_{n+1}}\right)\asymp\frac{{u}_p\left(1-\frac1{M_{n+1}}\right)}{M_{n+1}},
    \end{equation*}
and thus
    \begin{equation}\label{5.}
    M_{n+1}^{1-\frac1p}\widehat{\om}\left(1-\frac1{M_{n+1}}\right)^\frac1p\int_0^1t^{\frac{M_n}4}|f(t)|\,dt\lesssim \|f\|_{A^p_\om}.
    \end{equation}
Let now $k_0\in\N$ to be fixed later. Since $q\ge p$, by using \eqref{eq:sb3},
\eqref{5.},  H\"older's inequality and
\eqref{le:minfty}, we deduce
    \begin{equation}
    \begin{split}\label{eq:sb4}
   & \sum_{n=1}^{k_0}2^{-n}\left\|\Delta^\om_n\hg(f)\right\|^q_{H^q}
   \\ &\lesssim\sum_{n=1}^{k_0}2^{-n}
    \left(\int_0^1t^{\frac{M_n}{4}}|f(t)|\,dt\right)^q
\widehat{\om}\left(1-\frac1{M_{n+1}}\right)^{q(\frac1p-\frac1q)}M^{q(1-\frac1p)}_{n+1}
\\ & \lesssim  \|f\|^{q-p}_{A^p_\om}\sum_{n=1}^{k_0}2^{-n}
    \left(\int_0^1t^{\frac{M_n}{4}}|f(t)|\,dt\right)^p
M^{p-1}_{n+1}
\\
    &\lesssim\|f\|^{q-p}_{A^p_\om}\left(\int_0^1|f(t)|\,dt\right)^p\sum_{n=1}^{k_0}2^{-n}  M^{p-1}_{n+1}\\
    &\lesssim \|f\|^{q-p}_{A^p_\om}\int_0^1M_{\infty}^p(t, f)\,\widehat{\om}(t)\,dt\lesssim\left\|f\right\|^q_{A^p_\om},
    \end{split}
    \end{equation}
where the constants of comparison depend on $p$, $q$, $\om$ and
$k_0$.

Let now $\gamma_1$ be the constant appearing in
Lemma~\ref{le:Mncomparable}, and choose $k_0$ to be the smallest
natural number such that $r_{k_0}\ge \max\left\{
\frac{1}{2^{\gamma_1}}, \frac{1}{2} \right\}$ and
$2^{k_0\gamma_1}\ge 4$. Then, by \eqref{eq:sb3}, \eqref{5.} and
Lemma~\ref{le:Mncomparable}, we have
    \begin{equation}
    \begin{split}\label{eq:sb5}
    &\sum_{n=k_0+1}^\infty
    2^{-n}\left\|\Delta^\om_n\hg(f)\right\|^q_{H^q}\\
    &\lesssim\sum_{n=k_0+1}^\infty2^{-n}\left(\int_0^1t^{\frac{M_n}{4}}|f(t)|\,dt\right)^q\\
    &\quad\cdot\widehat{\om}\left(1-\frac1{M_{n+1}}\right)^{q(\frac1p-\frac1q)}M^{q(1-\frac1p)}_{n+1}\\
    &\lesssim \|f\|^{q-p}_{A^p_\om}\sum_{n=k_0+1}^\infty2^{-n}\left(\int_0^1t^{\frac{M_n}{4}}|f(t)|\,dt\right)^pM^{p-1}_{n+1}\\
    &\le C \|f\|^{q-p}_{A^p_\om}2^{(k_0+1)(\g_2(p-1)-1)} \sum_{j=0}^\infty 2^{-j}\left(\int_0^1 t^{\frac{2^{k_0\gamma_1} M_{j+1}}{4}}|f(t)|\,dt\right)^p M^{p-1}_{j+1}\\
    &\lesssim \|f\|^{q-p}_{A^p_\om}\sum_{j=0}^\infty 2^{-j}\left(\int_0^1 t^{M_{j+1}}|f(t)|\,dt\right)^p M^{p-1}_{j+1}.
    \end{split}
    \end{equation}
On the other hand, the M.~Riesz projection theorem and
Lemma~\ref{le:4} give
    \begin{equation*}
    \left\|\Delta^\om_n \left(\frac{1}{1-z}\right)
    \right\|^p_{H^p}\asymp  M^{p-1}_{n+1}.
    \end{equation*}
Now, by using Theorem~\ref{th:dec1} together with
Lemma~\ref{le:A}(ii), we get
    \begin{equation}
    \begin{split}\label{eq:sb6}
    \left\|\hti(f)\right\|^p_{A^p_\om}&\asymp \sum_{n=0}^\infty
    2^{-n}\left\|\Delta^\om_n\hti(f)\right\|^p_{H^p}
    \\ & \gtrsim \sum_{n=0}^\infty
    2^{-n} \left(\int_0^1 t^{M_{n+1}}|f(t)|\,dt\right)^p
    \left\|\Delta^\om_n \left(\frac{1}{1-z}\right) \right\|^p_{H^p}
    \\ & \gtrsim \sum_{n=0}^\infty
    2^{-n} \left(\int_0^1 t^{M_{n+1}}|f(t)|\,dt\right)^p
    M^{p-1}_{n+1}.
    \end{split}
    \end{equation}
So, by combining \eqref{eq:sb1}, \eqref{eq:sb2}, \eqref{eq:sb4},
\eqref{eq:sb5} and \eqref{eq:sb6}, we finally deduce
    \begin{equation*}
    \left\|\hg(f)\right\|^q_{A^q_\om}\lesssim \|g-g(0)\|^q_{\Lambda\left(q,\frac{1}{p},\ho^{\frac{1}{p}-\frac{1}{q}}\right)}\left(\|f\|^q_{A^p_\om}+\|f\|^{q-p}_{A^p_\om}\left\|\hti(f)\right\|^p_{A^p_\om}\right),
    \end{equation*}
which together with Corollary~\ref{hilbertapw} gives
$\left\|\hg(f)\right\|^q_{A^q_\om}\lesssim
\|g-g(0)\|^q_{\Lambda\left(q,\frac{1}{p},\ho^{\frac{1}{p}-\frac{1}{q}}\right)}\|f\|^q_{A^p_\om}$.
This finishes the proof of the sufficiency in the case $1<p\le
q<\infty$.
\medskip

Let now $1<q<p<\infty$ and assume that $g'\in
H\left(q,s,\ho_{s\left(1-\frac{1}{q}\right)}\right)$. By
\eqref{eq:sb1}, \eqref{eq:sb2}, Lemma~\ref{le:n14} and
H\"{o}lder's inequality, we deduce
    \begin{equation}
    \begin{split}\label{eq:sb10}
    &\left\|\hg(f)\right\|^q_{A^q_\om}
    \asymp\sum_{n=0}^\infty2^{-n}\left\|\Delta^\om_n\hg(f)\right\|^q_{H^q}\\
    &\lesssim \|f\|^q_{A^p_\om}+\sum_{n=1}^\infty2^{-n}\left\|\Delta^\om_n\hg(f)\right\|^q_{H^q}\\
    &\lesssim \|f\|^q_{A^p_\om}+\sum_{n=1}^\infty2^{-n} \left(\int_0^1t^{\frac{M_n}{4}}|f(t)|\,dt\right)^q\left\|\Delta^\om_n g'\right\|^q_{H^q}\\
    &\lesssim \|f\|^q_{A^p_\om}+\left[\sum_{n=1}^\infty2^{-n} M_n^{p-1}\left(\int_0^1t^{\frac{M_n}{4}}|f(t)|\,dt\right)^p\right]^{\frac{q}{p}}
    \left[\sum_{n=1}^\infty2^{-n} \frac{\left\|\Delta^\om_n g'\right\|^s_{H^q}}{M_n^{\left(1-\frac1p\right)s}}\right]^{1-\frac{q}{p}}.
    \end{split}
    \end{equation}
Corollary~\ref{co:mixtogamma} and the assumption $\om\in\R$ yield
    \begin{equation}
    \begin{split}\label{eq:sb11}
    \left[\sum_{n=1}^\infty2^{-n} \frac{\left\|\Delta^\om_n g'\right\|^s_{H^q}}{M_n^{\left(1-\frac1p\right)s}}\right]^{1-\frac{q}{p}}
    &\lesssim
    \left[\int_0^1 M^s_q(r,g')(1-r)^{\left(1-\frac1p\right)s}\om(r)\,dr\right]^{1-\frac{q}{p}}\\
    &\lesssim \|g'\|^q_{H\left(q,s,\ho_{s\left(1-\frac{1}{q}\right)}\right)}.
    \end{split}
    \end{equation}
Moreover, by arguing as in the previous case we obtain
    \begin{equation*}
    \begin{split}
    \left[\sum_{n=1}^\infty2^{-n} M_n^{p-1}\left(\int_0^1t^{\frac{M_n}{4}}|f(t)|\,dt\right)^p\right]^{\frac{q}{p}}
    &\lesssim \|f\|^q_{A^p_\om}+ \|\hti(f)\|^q_{A^p_\om}\lesssim\|f\|^q_{A^p_\om},
    \end{split}
    \end{equation*}
which together with \eqref{eq:sb10} and \eqref{eq:sb11} gives
    $$
    \left\|\hg(f)\right\|^q_{A^q_\om}
    \lesssim \|g'\|^q_{H\left(q,s,\ho_{s\left(1-\frac{1}{q}\right)}\right)}\|f\|^q_{A^p_\om}+\|f\|^q_{A^p_\om}.
    $$
The proof of the sufficiency is now complete.

\subsection{Test functions}

Before passing to the proof of the necessity part of
Theorem~\ref{mainth}, we will construct appropriate test
functions.  If $q<p$ we set up a family of functions
${Q_\r}\in A^p_\om$, depending on $g$, such that
    $$
    \lim_{\r\to1^-}\|Q_\r\|^p_{A^p_\om}\asymp\|g'\|_{H\left(q,s,\ho_{s\left(1-\frac{1}{q}\right)}\right)}.
    $$
In the case $q\ge p$ we will use the next result which can be
proved by using ideas from \cite[Lemma~1]{GaGiPeSis}.

\begin{lemma}\label{le:hg1}
Let $0<p-1<\gamma<\infty$ and $\om\in\I\cup\R$. Let
$E\subset(0,\infty)$ be a bounded set such that
$\rm{dist}$$(E,0)>0$. For $N\in\N$, let $a_N=1-\frac1N$, and
consider the functions
    \begin{equation*}
    \psi_{N,\om}(s)=\left[N^{\gamma+1}\om\left(S\left(a_N\right)\right)\right]^{-\frac{1}{p}}\int_0^1\frac{t^{sN}}{(1-a_Nt)^{\frac{\gamma+1}{p}}}\,dt,\quad
    s>0,
    \end{equation*}
and
    \begin{equation}\label{eq:psialpha2}
    \varphi_{N,\om}(s)=\frac{1}{\psi_{N,\om}(s)},\quad
    s>0.
    \end{equation}
Then the following assertions hold:
\begin{itemize}
\item[\rm(i)]$\psi_{N,\om}, \varphi_{N,\om} \in
C^\infty((0,\infty))$. \item[\rm(ii)] There exists a constant
$C=C(E)>0$ such that
    $$
    C^{-1} N^{-1}\om\left(S\left(a_N\right)\right)^{-\frac{1}{p}}\le|\psi_{N,\om}(s)|\le C N^{-1}\om\left(S\left(a_N\right)\right)^{-\frac{1}{p}},\quad s\in E,\quad N\to
    \infty.
    $$
\item[\rm(iii)] For each $m\in\N$, there exists a constant
$C=C(m,E)>0$ such that
    $$
    |\psi^{(m)}_{N,\om}(s)|\le CN^{-1}\om\left(S\left(a_N\right)\right)^{-\frac{1}{p}},\quad
    s\in E,\quad N\in\N.
    $$
\item[\rm(iv)]For each $m\in\N$, there exists a constant
$C=C(m,E)>0$ such that
    \begin{equation}\label{eq:psialpha3}
    |\varphi^{(m)}_{N,\om}(s)|\le C
    N\om\left(S\left(a_N\right)\right)^{\frac{1}{p}},\quad
    s\in E,\quad N\in\N.
    \end{equation}
\end{itemize}
\end{lemma}

Next, we will construct the test functions which will be used in
the proof of the case $q<p$. As usual, we write $f_\r(z)=f(\r z)$
for each $0\le\r<1$.

\begin{lemma}\label{testfunctionsq<p}
Let $1<q<p<\infty$, $\om\in\R\cap\Mp$ and $g\in H(\D)$  such that
$g'\in H\left(q,s,\ho_{s\left(1-\frac{1}{q}\right)}\right)$. Then
the functions
    \begin{equation*}
    \phi_\r(r)=\left(M_q(r,g_\r')(1-r)^{1-\frac{1}{q}}\right)^\frac{q}{p-q},\quad0<\r<1,
    \end{equation*}
and
    \begin{equation*}
    Q_\r(z)=\int_0^1\frac{\phi_\r(t)}{1-tz}\,dt,\quad
    z\in\D,
    \end{equation*}
satisfy
    \begin{equation}\label{9.}
    \begin{split}
    Q_\r(t)\gtrsim\phi_\r(t),\quad0\le t<1,
    \end{split}
    \end{equation}
and
    \begin{equation}\label{12.}
    \|Q_\r\|_{A^p_\om}^p\asymp\int_0^1\phi^p_\r(t)\,\widehat{\om}(t)\,dt<\infty.
    \end{equation}
\end{lemma}

\begin{proof}
Clearly, $Q_\r\in\H(\D)$ for all $0<\r<1$. Moreover,
    \begin{equation*}
    \begin{split}
    Q_\r(r)&\ge\int_r^1\frac{\phi_\r(t)}{1-tr}\,dt
    \ge\frac{M_q^\frac{q}{p-q}(r,g_\r')}{1-r^2}\int_r^1(1-t)^{\frac{q-1}{p-q}}\,dt
    \asymp\phi_\r(r),\quad 0\le r<1.
    \end{split}
    \end{equation*}
This and \eqref{le:minfty} give
 \begin{equation}\label{j15}
    \|Q_\r\|_{A^p_\om}^p\gtrsim\int_0^1\phi_\r^p(t)\,\widehat{\om}(t)\,dt.
    \end{equation}
Since $\om\in\R$ by the assumption, $\om$ is comparable to the
differentiable weight $\frac{\widehat{\om}(r)}{1-r}$, and hence
$n$ consecutive applications of \cite[Theorem~1.1]{PavP} give
    \begin{equation}\label{LPn}
    \|f\|_{A^p_\om}^p\asymp\sum_{j=0}^{n-1}|f^{(j)}(0)|^p+\int_\D|f^{(n)}(z)|^p(1-|z|)^{np}\,\om(z)\,dA(z),\quad
    f\in\H(\D).
    \end{equation}
Now
    $$
    Q_\r^{(n)}(z)=n!\int_0^1\frac{t^n\phi_\r(t)}{(1-tz)^{n+1}}\,dt,
    $$
and so Minkowski's inequality in continuous form yields
    \begin{equation*}
    \begin{split}
    M_p(r,Q_\r^{(n)})&=\left(\frac{n!}{2\pi}\int_0^{2\pi}\left|\int_0^1\frac{\phi_\r(t)t^n}{(1-tre^{i\t})^{n+1}}\,dt\right|^p\,d\t\right)^\frac1p\\
    &\lesssim\int_0^1\phi_\r(t)\left(\int_0^{2\pi}\frac{d\t}{|1-tre^{i\t}|^{np+p}}\right)^\frac1p\,dt
    \asymp\int_0^1\frac{\phi_\r(t)}{(1-tr)^{n-\frac1p+1}}\,dt.
    \end{split}
    \end{equation*}
Choose now $n\in\N$ such that $n-\frac1p-\frac{q-1}{p-q}>0$. Then
    \begin{equation*}
    \begin{split}
    \int_0^r\frac{\phi_\r(t)}{(1-tr)^{n-\frac1p+1}}\,dt
    &\le
    M_q(r,g_\r')^\frac{q}{p-q}\int_0^r\frac{(1-t)^{\frac{q-1}{p-q}}}{(1-tr)^{n-\frac1p+1}}\,dt\\
    &\le M_q(r,g_\r')^\frac{q}{p-q}\int_0^r\frac{dt}{(1-tr)^{n-\frac1p+1-\frac{q-1}{p-q}}}
    \asymp\frac{\phi_\r(r)}{(1-r)^{n-\frac1p}}.
    \end{split}
    \end{equation*}
Moreover, since $\om\in\Mp$, by applying \cite[Theorem~2]{Muckenhoupt1972}, with
    $$
    U^{p}(x)=\left\{
        \begin{array}{cl}
        \frac{\widehat{\om}(x)}{(1-x)^p}, &   0\le x<1\\
        0, & x\ge1
        \end{array}\right.,
    $$
and
    $$
    V^{p}(x)=\left\{
        \begin{array}{cl}
        \widehat{\om}(x), &   0\le x<1\\
        0, & x\ge1
        \end{array}\right.,
    $$
we deduce
    \begin{equation*}
    \begin{split}
    &\int_0^1M_p^p(r,Q_\r^{(n)})(1-r)^{np}\,\om(r)\,dr\\
    &\lesssim\int_0^1\left(\left(\int_0^r+\int_r^1\right)\frac{\phi_\r(t)}{(1-tr)^{n-\frac1p+1}}\,dt\right)^p(1-r)^{np}\,\om(r)\,dr\\
    &\lesssim\int_0^1\phi_\r^p(r)(1-r)\om(r)\,dr
    +\int_0^1\left(\int_r^1\phi_\r(t)\,dt\right)^p\frac{\om(r)}{(1-r)^{p-1}}\,dr\\
    &\lesssim\int_0^1\phi_\r^p(r)\widehat{\om}(r)\,dr
    +\int_0^1\left(\int_r^1\phi_\r(t)\,dt\right)^p\frac{\widehat{\om}(r)}{(1-r)^{p}}\,dr\\
    &\lesssim\int_0^1\phi_\r^p(r)\widehat{\om}(r)\,dr<\infty,
    \end{split}
    \end{equation*}
which together with \eqref{LPn} and \eqref{j15} gives \eqref{12.}.
This finishes the proof.
\end{proof}

\subsection{Necessity}

First we deal with the case $1<p\le q<\infty$. Let
$g(z)=\sum_{k=0}^\infty b_kz^k$ be the Maclaurin series of $g$. By
Lemma~\ref{le:Mncomparable} there exists a positive constant
$B_2=B_2(\om)$ such that
    \begin{equation}\label{eq:j16}
    1\le \frac{M_{n+1}}{M_n}\le B_2, \quad n\in\N.
    \end{equation}
Let us consider the functions $\psi_{M_n, \om}$  and
$\varphi_{M_n, \om}=\frac{1}{\psi_{M_n, \om}}$ defined in Lemma
\ref{le:hg1}. For each $n\in\N$, we can find  a
$C^\infty$-function
 $\Phi_{M_n}:\mathbb{R}\to\mathbb{C}$ with $\supp\left(\Phi_{M_n}\right)\subset \left(\frac{1}{2}, 2B_2 \right)$, satisfying
    \begin{equation}\label{eq:nh1}
    \Phi_{M_n}(s)=\varphi_{M_n, \om}(s),\quad 1\le s\le B_2,
    \end{equation}
and such that, by using part Lemma~\ref{le:hg1}(iv), for each
$m\in\N$ there exists a constant $C=C(m)>0$ for which
    \begin{equation}\label{eq:nh2}
    |\Phi^{(m)}_{M_n}(s)|\le C M_n\om\left( S\left(a_{M_n}\right)\right)^{1/p},\quad s\in\mathbb{R},\quad n\in\N.
\end{equation}
In particular, by \eqref{eq:nh2} and Lemma~\ref{le:hg1}(ii), we
have
    \begin{equation}\label{eq:nh2-1}
    A_{\Phi_{M_n}}=\max_{s\in\mathbb{R}}|\Phi_{M_n}(s)|+\max_{s\in\mathbb{R}}|\Phi_{M_n}''(s)|\lesssim M_n\om\left( S\left(a_{M_n}\right)\right)^{1/p}.
    \end{equation}
Let us now consider the functions
  \begin{equation}\label{eq:fN}
    f_{M_n}(z)=\frac{1}{\left(M_n^{\gamma+1}\om\left( S\left(a_{M_n}\right)\right)\right)^{1/p}}\frac{1}{(1-a_{M_n}z)^{\frac{\gamma+1}{p}}}
    ,\quad z\in\D,\quad n\in\N,
    \end{equation}
where $\g>\max\{\g_0,p-1\}$ and $\g_0=\g_0(\om)>0$ is from
Lemma~\ref{Lemma:Zhu-type}. The $A^p_\om$-norms of the functions
$f_{M_n}$ are uniformly bounded by Lemma~\ref{Lemma:Zhu-type}.
Therefore
    $$
    \sup_{n\in\N}\|\hg(f_{M_n})\|_{A^q_\om}\lesssim \|\H_g\|_{\left(A^p_\om,A^q_\om\right)}<\infty
    $$
by the hypothesis. This together with Theorem~\ref{th:cesaro}(iv)
and \eqref{eq:nh2-1} implies
    \begin{equation}
    \begin{split}\label{eq:nh3}
    \|W_{M_n}^{\Phi_{M_n}}\ast \hg(f_{M_n})\|_{A^q_\om}
    &\lesssim A_{\Phi_{M_n}}\|\hg(f_{M_n})\|_{A^q_\om}\\
    &\lesssim \|\H_g\|_{\left(A^p_\om,A^q_\om\right)} M_n\om\left( S\left(a_{M_n}\right)\right)^{1/p}.
    \end{split}
    \end{equation}
On the other hand, bearing in mind the M.~Riesz
projection theorem, \eqref{eq:j16}, \eqref{eq:nh1}, \eqref{eq:fN} and
\eqref{eq:psialpha2}, we deduce
    \begin{align*}
    &\|W_{M_n}^{\Phi_{M_N}}\ast \hg(f_{M_n})\|_{A^q_\om}\\
    &\gtrsim \left\| \sum_{M_n\le k\le M_{n+1}-1}
(k+1)b_{k+1}\left(\int_0^1t^kf_{M_n}(t)\,dt\right)\Phi_{M_n}\left
(\frac{k}{M_n}\right )z^k\right\|_{A^q_\om}
\\ & = \left\| \sum_{M_n\le k\le M_{n+1}-1}
(k+1)b_{k+1}\left(\int_0^1t^kf_{M_n}(t)\,dt\right)\varphi_{M_n,\om}\left
(\frac{k}{M_n}\right )z^k \right \|_{A^q_\om}
\\ &  =\|\Delta_n^\om g'\|_{A^q_\om},
    \end{align*}
which together with \eqref{eq:nh3},  Lemma~\ref{le:eqno},
Lemma~\ref{le:condinte} and \eqref{rn} gives
    \begin{equation*}
    \begin{split}
    \|\Delta_n^\om g'\|_{H^q}
    &\lesssim\|\H_g\|_{\left(A^p_\om,A^q_\om\right)}{M_n}^{1-\frac{1}{p}}
    \left(\widehat{\om}\left(1-\frac1{M_n}\right)\right)^{\frac1p-\frac1q}\\
    &\lesssim\|\H_g\|_{\left(A^p_\om,A^q_\om\right)}{M_n}^{1-\frac{1}{p}}
    2^{-n\left(\frac1p-\frac1q\right)}.
    \end{split}\end{equation*}
Finally, Corollary~\ref{Lemma:q>p}(i) implies
$g\in\Lambda\left(q, \frac1p,\ho^{\frac1p-\frac1q}\right)$ and
    $$
    \|g-g(0)\|_{\Lambda\left(q, \frac1p,\ho^{\frac1p-\frac1q}\right)}
    \lesssim \|\H_g\|_{\left(A^p_\om,A^q_\om\right)}.
    $$

Let now $1<q<p<\infty$ and assume that $\H_g:A^p_\om\to A^q_\om$
is bounded. Let $\{\phi_\r\}$ and  $\{Q_\r\}$ be the families of
functions considered in Lemma~\ref{testfunctionsq<p}. Since each
$Q_\r$ is increasing on $[0,1)$, Lemma~\ref{le:Mncomparable} gives
    \begin{equation*}
    \int_0^1t^{M_n}Q_\r(t)\,dt\asymp\int_0^1t^{M_{n+1}}Q_\r(t)\,dt,\quad
    n\in\N.
    \end{equation*}
So, Theorem~\ref{th:dec1},  Lemma~\ref{le:A} and Lemma~\ref{le:4}
imply
    \begin{equation}\label{eq:j20}
    \begin{split}
    \|\H_{g_\r}(Q_\r)\|_{A^q_\om}^q
    &\asymp\sum_{n=0}^\infty 2^{-n}
    \|\Delta^{\om}_n \H_{g_\r}(Q_\r)\|_{H^q}^q\\
    &\asymp\sum_{n=0}^\infty 2^{-n}\left(\int_0^1t^{M_n}Q_\r(t)\,dt\right)^q
    \|\Delta^{\om}_n g_\r'\|_{H^q}^q\\
    &\asymp\sum_{n=0}^\infty 2^{-n}\left(\int_0^1t^{M_n}Q_\r(t)\,dt\right)^q
    M_q^q\left(1-\frac1{M_{n+1}},\Delta^{\om}_n g_\r'\right).
    \end{split}
    \end{equation}
Since $Q_\r$ is increasing on $[0,1)$, \eqref{9.}, the M.~Riesz
projection theorem and Lemma~\ref{le:4} yield
    \begin{equation}\label{eq:j21}
    \begin{split}
    \int_0^1t^{M_n}Q_\r(t)\,dt
    &\gtrsim\frac{Q_\r\left(1-\frac1{M_{n+1}}\right)}{M_{n+1}}\\
    &\gtrsim \frac{M_q^\frac{q}{p-q}\left(1-\frac1{M_{n+1}},g_\r'\right)}{M_{n+1}^{1+\frac{q-1}{p-q}}}
    \gtrsim\frac{\|\De_n^\om
    g_\r'\|_{H^q}^\frac{q}{p-q}}{M_{n+1}^{\frac{p-1}{p-q}}}.
    \end{split}
    \end{equation}
So, by combining \eqref{eq:j20}, \eqref{eq:j21},
Corollary~\ref{co:mixtogamma} and Lemma~\ref{testfunctionsq<p}, we
obtain
    \begin{equation}\label{eq:j22}
    \begin{split}
    \|\H_{g_\r}(Q_\r)\|_{A^q_\om}^q
    &\gtrsim\sum_{n=0}^\infty 2^{-n}\frac{1}{M_{n+1}^{q\frac{p-1}{p-q}}}\|\De_n^\om
    g_\r'\|_{H^q}^\frac{pq}{p-q}\\
    &\asymp \int_0^1 M^s_q(r,g_\r')(1-r)^{\left(1-\frac1p\right)s}\om(r)\,dr\\
    &\asymp\int_0^1\phi_\r^p(r)\,\widehat{\om}(r)\,dr\asymp\|Q_\r\|_{A^p_\om}^p.
    \end{split}
    \end{equation}
Further, \eqref{eq:j20} yields
    $$
    \|\H_{g_\r}(Q_\r)\|_{A^q_\om}^q\le C \|\H_{g}(Q_\r)\|_{A^q_\om}^q,\quad 0<\r<1,
    $$
where $C$ does not depend on $\r$. This together with
\eqref{eq:j22} gives
    \begin{equation*}
    \begin{split}
    \infty&>\|\H_g\|^q_{\left(A^p_\om,A^q_\om\right)} \ge
    \frac{\|\H_{g}(Q_\r)\|_{A^q_\om}^q}{\|Q_\r\|_{A^p_\om}^q} \gtrsim
    \frac{\|\H_{g_\r}(Q_\r)\|_{A^q_\om}^q}{\|Q_\r\|_{A^p_\om}^q}\\
    &\gtrsim \|Q_\r\|_{A^p_\om}^{p-q}
    \asymp\|g_\r'\|^q_{H\left(q,s,\ho_{s\left(1-\frac1q\right)}\right)},
    \end{split}
    \end{equation*}
so, by letting $\r\to 1^-$, we deduce
    $$
    \|\H_g\|^q_{\left(A^p_\om,A^q_\om\right)}\gtrsim
    \|g'\|^q_{H\left(q,s,\ho_{s\left(1-\frac1q\right)}\right)}.
    $$
This finishes the proof.

\section{Compact and Hilbert-Schmidt  operators}\label{compact}

\subsection{Compactness}\label{compactness}

The main objective of this section is to prove the following
result.

\begin{theorem}\label{mainthcompact}
Let $1<p,q<\infty$, $\om\in\R\cap\Mp$ and $g\in \H(\D)$.
    \begin{enumerate}
    \item[\rm(i)] If $1<p\le q<\infty$, then $\hg: A^p_\om \to A^q_\om$ is compact if and only if
    $g\in \lambda\left(q,\frac{1}{p}, \ho^{\frac{1}{p}-\frac{1}{q}}\right)$.
    \item[\rm(ii)] If $1<q<p<\infty$, then $\H_g:A^p_\om\to A^q_\om$ is
    compact if and only it is bounded.
    \end{enumerate}
\end{theorem}

We will need the following lemma, which can be easily proved by
using \eqref{99}, H\"{o}lder's inequality and \eqref{le:minfty}.

\begin{lemma}\label{le:c1}
Let $1<p<\infty$ and let $\om$ be a radial weight such that
\eqref{99} is satisfied. Let $\{f_j\} _{j=1}^\infty $ be a
sequence in $A^p_\om$ such that
$\sup_{j}\|f_j\|_{A^p_\om}=K<\infty$ and $f_j\to 0$, as
$j\to\infty $, uniformly on compact subsets of\, $\D$. Then the
following assertions hold:
\begin{itemize}
\item[\rm(i)] $\lim_{j\to\infty}\int_0^1 |f_j(t)|\,dt=0$;
\item[\rm(ii)] $\mathcal H_g(f_j)\to 0$, as $j\to\infty$,
uniformly on compact subsets of $\D$ for each $g\in\H(\D)$.
\end{itemize}
\end{lemma}

Next, we remind the reader that for $\om\in\I\cup\R$, the norm
convergence in $A^p_\om$ implies the uniform convergence on
compact subsets of $\D$ by \cite[Lemma~2.5]{PelRat}. This fact and
Lemma~\ref{le:c1} are the key tools in the proof of the following
result whose proof will be omitted.

\begin{lemma}\label{le:c12} Let $1<p<\infty$, $0<q<\infty$ and $\om\in\I\cup\R$ such that \eqref{99} is satisfied, and let $g\in \H(\D)$.
Then the following conditions are equivalent:
\begin{itemize}\item[\rm(i)] $\hg: A^p_\om\to A^q_\om$ is compact;
\item[\rm(ii)] For each sequence $\{f_j\}_{j=1}^\infty $ in
$A^p_\om$ for which
    \begin{equation}\label{eq:com1}
    \sup_{k}\|f_j\|_{A^p_\om}=K <\infty
    \end{equation}
and
    \begin{equation}\label{eq:com2}
    f_j\to 0, \,\text{as $j\to\infty $,\, uniformly on compact subsets
    of $\D$},
    \end{equation}
we have $\lim_{j\to\infty}\|\mathcal H_g(f_j)\|_{A^q_\om}=0$.
\end{itemize}
\end{lemma}

\medskip

\noindent\emph{Proof of} Theorem~\ref{mainthcompact}. (i). Assume
first that $\hg:A^p_\om\to A^q_\om$ is compact. Let
$\{f_{M_n}\}_{n=0}^\infty$ be the family of test functions
    $$
    f_{M_n}(z)=\frac{1}{\left(M_n^{\gamma+1}\om\left( S\left(a_{M_n}\right)\right)\right)^{1/p}}
    \frac{1}{(1-a_{M_n}z)^{\frac{\gamma+1}{p}}},\quad z\in\D,\quad n\in\N,
    $$
considered in \eqref{eq:fN}. If $\g$ is large enough,
Lemma~\ref{Lemma:Zhu-type} ensures that $\{f_{M_n}\}_{n=0}^\infty$
satisfies \eqref{eq:com1}. Now the proof of
Lemma~\cite[Lemma~1.1]{PelRat} shows that
$\lim_{|a|\to1^-}\frac{(1-|a|)^{\gamma}}{\widehat{\om}(|a|)}=0$,
if $\gamma>0$ is again large enough. So, if $\gamma$ is fixed
appropriately, then
    $$
    \lim_{n\to\infty}f_{M_n}(z)=\lim_{n\to\infty}
    \frac{1}{\left({M_n}^{\gamma}\widehat{\om}\left(1-\frac{1}{M_n}\right)\right)^{1/p}}\frac{1}{(1-a_{M_n}z)^{\frac{\gamma+1}{p}}}=0
    $$
uniformly on compact subsets of $\D$. Thus
$\{f_{M_n}\}_{n=0}^\infty$ satisfies \eqref{eq:com2}. Therefore
Lemma~\ref{le:c12} implies
    \begin{equation}\label{eq:c1}
    \lim_{n\to\infty}\|\mathcal H_g(f_{{M_n}})\|_{A^p_\om}=0.
    \end{equation}
Next, a careful inspection of the proof of the necessity part of
Theorem~\ref{mainth} reveals the inequalities
    \begin{equation*}
    \begin{split}
    \|\Delta_n^\om g'\|_{H^q}&\lesssim \|\hg(f_{M_n})\|_{A^q_\om} {M_n}^{1-\frac{1}{p}}\widehat{\om}\left(1-\frac{1}{M_n}\right)^{\frac1p-\frac1q}\\
    &\lesssim\|\hg(f_{M_n})\|_{A^q_\om}{M_n}^{1-\frac{1}{p}} 2^{-n\left(\frac1p-\frac1q\right)},\quad n\in\N.
    \end{split}
    \end{equation*}
Finally, \eqref{eq:c1} and Corollary~\ref{Lemma:q>p}(ii) imply
$g\in\lambda\left(q, \frac1p,\ho^{\frac1p-\frac1q}\right)$.

Conversely, let $\ep>0$ and $g\in  \lambda\left(q,
\frac1p,\ho^{\frac1p-\frac1q}\right)$. Then there exists
$r_0=r_0(\e)\in[0,1)$ such that
    \begin{equation}\label{eq:c3}
    M_q^q(r,g')\le\e\frac{\ho^{\frac{q}{p}-1}(r)}{(1-r)^{q(1-\frac1p)}},\quad r\ge r_0.
    \end{equation}
Let now $k_0$ be the integer which appears in the proof of the
sufficiency part of Theorem~\ref{mainth}, and choose $n_0\ge k_0$,
such that
    \begin{equation}\label{eq:c5}
    1-\frac{1}{M_{n+1}}\ge r_0, \quad n\ge n_0.
    \end{equation}
Let $\{f_j\}$ be a sequence of analytic functions in $\D$
satisfying \eqref{eq:com1} and \eqref{eq:com2}. By arguing as in
the proof of Theorem~\ref{mainth} and bearing in mind that
$\lambda \left(q,\frac{1}{p},\ho^{\frac1p-\frac1q}\right)\subset
\Lambda \left(q,\frac{1}{p},\ho^{\frac1p-\frac1q}\right)$, we
deduce
    \begin{equation}
    \begin{split}\label{eq:c4}
    \sum_{n=0}^{n_0}2^{-n}\left\|\Delta^\om_n\hg(f_j)\right\|^q_{H^q}
    &\lesssim\|f_j\|^{q-p}_{A^p_\om}\sum_{n=0}^{n_0} 2^{-n} \left(\int_0^1t^{\frac{M_n}{4}}|f_j(t)|\,dt\right)^p M^{p-1}_{n+1}\\
    &\lesssim \left(\int_0^1 |f_j(t)|\,dt\right)^p,
    \end{split}
    \end{equation}
where the constants of comparison depend on $g$, $p$, $\om$, $K$  and
$n_0$.

On the other hand, an analogous reasoning to that in
\eqref{eq:sb3}, \eqref{eq:c3}, \eqref{eq:c5}, \eqref{5.} and
\eqref{eq:com1} give
    \begin{equation*}
    \begin{split}
    \left\|\Delta^\om_n\hg(f_j)\right\|^q_{H^q}
    &\lesssim\left(\int_0^1 t^{\frac{M_n}{4}}|f_j(t)|\,dt\right)^q\left\|\Delta^\om_n
    g'\right\|^q_{H^q}\\
    &\lesssim \ep \|f_j\|^{q-p}_{A^p_\om}\left(\int_0^1t^{\frac{M_n}{4}}|f_j(t)|\,dt\right)^p M^{p-1}_{n+1}\\
    &\le\ep K^{q-p}\left(\int_0^1 t^{\frac{M_n}{4}}|f_j(t)|\,dt\right)^p
    M^{p-1}_{n+1},
    \end{split}
    \end{equation*}
so bearing in mind that $n_0\ge k_0$, Corollary~\ref{hilbertapw},
\eqref{eq:com1} and following the proof of Theorem~\ref{mainth},
we get
    \begin{equation*}
    \begin{split}
    \sum_{n=n_0+1}^\infty 2^{-n}\left\|\Delta^\om_n\hg(f_j)\right\|^q_{H^q}
    &\lesssim \ep K^{q-p}\left\|\hti(f_j)\right\|^p_{A^p_\om}
    \lesssim \ep K^{q-p}\|f_j\|^p_{A^p_\om}\lesssim \ep K^q.
\end{split}
\end{equation*}
This together with \eqref{eq:sb1}, \eqref{eq:c4} and
Lemma~\ref{le:c1}, imply
    \begin{equation*}
    \begin{split}
    \lim_{j\to\infty}\left\|\hg(f_j)\right\|^p_{A^p_\om}
    \lesssim&\lim_{j\to\infty}\left( \left(\int_0^1 |f_j(t)|\,dt\right)^p+\ep K^q\right)
    \lesssim \ep K^q.
    \end{split}
    \end{equation*}
Since $\ep>0$ is arbitrary, Lemma~\ref{le:c12} shows that
$\H_g:A^p_\om\to A^q_\om$ is compact.

(ii). By Theorem~\ref{mainth}(ii) it is enough to show that
$\H_g:A^p_\om\to A^q_\om$ is compact if $g'\in
H(q,s,\widehat{\om}^{s(1-\frac1q)})$. Let $\{f_j\}$ be a sequence
of analytic functions in $\D$ satisfying \eqref{eq:com1} and
\eqref{eq:com2}. Let $\ep>0$ be given. By the proof of
Corollary~\ref{co:mixtogamma}, there exists $n_0\in\N$ such that
    \begin{equation*}
    \begin{split}
    \left[\sum_{n=n_0}^\infty2^{-n}\frac{\left\|\Delta^\om_n g'\right\|^s_{H^q}}{M_n^{\left(1-\frac1p\right)s}}\right]^{1-\frac{q}{p}}
    &<\ep
    \end{split}
    \end{equation*}
for all $n\ge n_0$. By H\"{o}lder's inequality, \eqref{eq:com1}
and a reasoning similar to that in the proof of
Theorem~\ref{mainth}, we obtain
    \begin{equation*}
    \begin{split}
    &\sum_{n=n_0}^\infty2^{-n} \left(\int_0^1t^{\frac{M_n}{4}}|f_j(t)|\,dt\right)^q\left\|\Delta^\om_n g'\right\|^q_{H^q}\\
    &\le\left[\sum_{n=n_0}^\infty2^{-n} M_n^{p-1}\left(\int_0^1t^{\frac{M_n}{4}}|f_j(t)|\,dt\right)^p\right]^{\frac{q}{p}}
    \left[\sum_{n=n_0}^\infty2^{-n} \frac{\left\|\Delta^\om_n g'\right\|^s_{H^q}}{M_n^{\left(1-\frac1p\right)s}}\right]^{1-\frac{q}{p}}\\
    &\lesssim \ep\left[\sum_{n=1}^\infty2^{-n} M_n^{p-1}\left(\int_0^1t^{\frac{M_n}{4}}|f_j(t)|\,dt\right)^p\right]^{\frac{q}{p}}\\
    &\lesssim \ep \left(\|f_j\|^q_{A^p_\om}+ \|\hti(f_j)\|^q_{A^p_\om}\right)
    \lesssim \ep \|f_j\|^q_{A^p_\om}
    \lesssim \ep K^q.
    \end{split}
    \end{equation*}
On the other hand, \eqref{eq:sb2} and Lemma~\ref{le:n14} give
    $$
    \sum_{n=0}^{n_0-1}2^{-n}\left\|\Delta^\om_n\hg(f)\right\|^q_{H^q}
    \lesssim \left(\int_0^1|f_j(t)|\,dt\right)^q\sum_{n=0}^{n_0-1}\|\Delta_n^\om g'\|_{H^q}^q,
    $$
and hence  by \eqref{eq:sb1} and Lemma~\ref{le:n14}
    \begin{equation*}
    \begin{split}
    &\left\|\hg(f_j)\right\|^q_{A^q_\om}\\
    &\lesssim\sum_{n=0}^{n_0-1}2^{-n}\left\|\Delta^\om_n\hg(f_j)\right\|^q_{H^q}
    +\sum_{n=n_0}^\infty2^{-n}\left(\int_0^1t^{\frac{M_n}{4}}|f_j(t)|\,dt\right)^q\left\|\Delta^\om_n
    g'\right\|^q_{H^q}\\
    &\lesssim \left(\int_0^1|f_j(t)|\,dt\right)^q\sum_{n=0}^{n_0-1}\|\Delta_n^\om g'\|_{H^q}^q+\ep K^q.
    \end{split}
    \end{equation*}
Finally, since $\ep>0$ is arbitrary and $n_0\in\N$ fixed,
Lemma~\ref{le:c1} gives
    $$
    \lim_{j\to\infty } \left\|\hg(f_j)\right\|_{A^q_\om}=0,
    $$
which together with Lemma \ref{le:c12} finishes the proof.
\hfill$\Box$

\subsection{Hilbert-Schmidt operators}

In this section we offer a characterization of those symbols $g$
for which the operator $\mathcal{H}_g$ is Hilbert-Schmidt on
$A^2_\om$, where $\om\in\R\cap\mathcal{M}_2$. Recall that the
classical \emph{Dirichlet space} consists of those functions
$g\in\H(\D)$ for which
    \begin{displaymath}
    \|g\|_{\mathcal{D}}^2=
    |g(0)|^2+\int_{\D}|g'(z)|^2\,dA(z)<\infty.
    \end{displaymath}

\begin{theorem}\label{th:sachatten}
Let $\om\in\R\cap\mathcal{M}_2$ and $g\in\H(\D)$. Then
$\mathcal{H}_g$ is Hilbert-Schmidt on $A^2_\om$ if and only if
$g\in\mathcal{D}$.
\end{theorem}

\begin{proof}
Denote
    $$
    \om_n=\int_0^1 r^{2n+1}\om(r)\,dr,\quad e_n(z)=\frac{z^n}{\sqrt{2}\,\om^{1/2}_n},\quad n\in\N,
    $$
and consider the basis $\left\{e_n\right\}$ of $A^2_\om$. If
$g(z)=\sum_0^{\infty} b_kz^k\in \H(\D)$,
then
    $$
    \left\|\hg(e_n)\right\|^2_{A^2_\om} =\frac{1}{2\om_n}\sum_{k=0}^{\infty}
    \frac{(k+1)^2|b_{k+1}|^2\om_k}{(n+k+1)^2}.
    $$
We claim that
    \begin{equation}\label{eq:suma}
    \sum_{n=0}^{\infty}
    \frac{1}{(n+k+1)^2\om_n}\asymp \frac{1}{(k+1)\om_k},\quad k\in\N.
    \end{equation}
So, assuming this for a moment, we deduce
    \begin{align*}
    \sum_{n=0}^{\infty}\left\|\hg(e_n)\right\|^2_{A^2_\om}& =\sum_{n=0}^{\infty}\frac{1}{2\om_n}\sum_{k=0}^{\infty}
    \frac{(k+1)^2|b_{k+1}|^2\om_k}{(n+k+1)^2}\\
    &=\frac12\sum_{k=0}^{\infty}(k+1)^2|b_{k+1}|^2\om_k
    \sum_{n=0}^{\infty}
    \frac{1}{(n+k+1)^2\om_n}\\
    &\asymp \sum_{k=0}^{\infty}(k+1)|b_{k+1}|^2 \asymp
    \|g-g(0)\|_{\mathcal{D}}^2,
    \end{align*}
which proves the assertion. It remains to prove \eqref{eq:suma}.
Clearly,
    \begin{equation}\label{eq:suma3}
    \sum_{n=k+1}^{\infty} \frac{1}{(n+k+1)^2\om_n}\ge
    \frac{1}{\om_k}\sum_{n=k+1}^{\infty} \frac{1}{(n+k+1)^2}\asymp
    \frac{1}{(k+1)\om_k},\quad k\in\N.
    \end{equation}
On the other hand,
    \begin{equation}\label{eq:suma2}
    \sum_{n=0}^{k}
    \frac{1}{(n+k+1)^2\om_n}\le \frac{1}{\om_k}\sum_{n=0}^{k}
    \frac{1}{(n+k+1)^2}\asymp \frac{1}{(k+1)\om_k},\quad k\in\N.
    \end{equation}
Moreover, since $\om\in\mathcal{M}_2$, Lemma~\ref{Lemma:u_p}
yields
    $$
    \int_r^1 \frac{dt}{\widehat{\om}(t)}\asymp
    \frac{(1-r)}{\widehat{\om}(r)},
    $$
which together with Lemma~\ref{le:condinte}(i) and (iv) gives
    \begin{equation*}
    \begin{split}
    \sum_{n=k+1}^{\infty}\frac{1}{(n+k+1)^2\om_n}
    &\le \sum_{n=k+1}^{\infty}\frac{1}{(n+1)^2\om_n}
    \asymp\sum_{n=k+1}^{\infty}\frac{1}{\widehat{\om}\left(1-\frac{1}{n}\right)}
    \int_{1-\frac{1}{n}}^{1-\frac{1}{n+1}}dt\\
    &\le\int_{1-\frac{1}{k+1}}^1\frac{dt}{\widehat{\om}(t)}
    \asymp\frac{1}{(k+1)\widehat{\om}\left(1-\frac{1}{k+1}\right)}
    \asymp\frac{1}{(k+1)\om_k}.
\end{split}\end{equation*}
This combined with \eqref{eq:suma2} and \eqref{eq:suma3} yields
\eqref{eq:suma} and finishes the proof.
\end{proof}

\section{Further results} \label{further}

\subsection{Descriptions of weighted spaces.}

In view of \cite[Theorem~5.4]{Duren1970} it is natural to expect
that, under appropriate assumptions, the space
$\Lambda\left(q,\a,\widehat{\om}^\eta\right)$ could be
characterized by a weighted $q$-mean Lipschitz condition. We show
that this is indeed the case of those spaces to which the
containment of the symbol $g$ characterizes the boundedness of
$\H_g:A^p_\om\to A^q_\om$ when $1<p\le q<\infty$ and
$\om\in\R\cap\Mp$.

\begin{proposition}\label{pr:Lambdafrontera}
Let $1<q,p<\infty$, $\eta\in\left[0,\frac{1}{p}\right)$,
$\om\in\R\cap\Mp$  and $g\in\H(\D)$. The the following assertions
hold:
\begin{itemize}
\item[\rm(i)]
$g\in\Lambda\left(q,\frac{1}{p},\widehat{\om}^\eta\right)$ if
and only if $g\in H^q$ and
    $$
    \sup_{0<h\le t}\left(\int_0^{2\pi}
    |g(e^{i(\theta+h)})-g(e^{i\theta})|^q \frac{d\theta}{2\pi}\right)^{1/q}=\og(t^{\frac{1}{p}}\ho^{\eta}(1-t)), \quad t\to
    0.
    $$
Moreover,
    $$
    \|g\|_{\Lambda\left(q,\frac{1}{p},\widehat{\om}^\eta\right)}\asymp |g(0)|+
    \sup_{0<h\le t}\frac{\left(\int_0^{2\pi}
    |g(e^{i(\theta+h)})-g(e^{i\theta})|^q \frac{d\theta}{2\pi}\right)^{1/q}}{t^{\frac{1}{p}}\ho^{\eta}(1-t)}.$$
 \item[\rm(ii)]$g\in\lambda(q,\frac{1}{p},\widehat{\om}^\eta)$ if and only if $g\in H^q$ and $$
    \sup_{0<h\le t}\left(\int_0^{2\pi}
    |g(e^{i(\theta+h)})-g(e^{i\theta})|^q \frac{d\theta}{2\pi}\right)^{1/q}=\op(t^{\frac{1}{p}}\ho^{\eta}(1-t)), \quad t\to
    0.
    $$
\end{itemize}
\end{proposition}

\begin{proof}
The proof of (i) consists  of a direct application of
\cite[Theorem~2.1(i)]{BlaSoaJMAA1990}. First, observe that if $g\in\Lambda\left(q,\frac{1}{p},\widehat{\om}^\eta\right)$, then $g\in H^q$.
Now, if we choose
    $$
    \vr(t)=t^{\frac{1}{p}}\ho^{\eta}(1-t), \quad 0\le t< 1,
    $$
it suffices to show that $\vr$ satisfies both, the {\em{Dini}
condition}
    \begin{equation}\label{Dini}
    \int_0^t \frac{\vr(s)}{s}\,ds\lesssim\vr(t), \quad 0<t<1,
    \end{equation}
and the {\em $b_1$-condition}
    \begin{equation}\label{b1}
    \int_t^1 \frac{\vr(s)}{s^2}\,ds\lesssim\frac{\vr(t)}{t},\quad 0<t< 1.
    \end{equation}
By Lemma~\ref{le:condinte}(i), we deduce the inequality
    \begin{equation*}
    \begin{split}
    \int_0^t s^{\frac{1}{p}-1}\left(\frac{\ho(1-s)}{\ho(1-t)}\right)^\eta\,ds
    \le\frac{1}{t^{\a\eta}}\int_0^t s^{\frac{1}{p}+\a\eta-1}\,ds
    =\frac{t^{\frac{1}{p}}}{\frac{1}{p}+\a\eta},
    \end{split}
    \end{equation*}
which is equivalent to \eqref{Dini}. Moreover, by using the fact
that $\frac{\ho(r){\frac{1}{p}}}{(1-r)^{1-\frac{1}{p}}}$ is
essentially increasing (see the proof of
Lemma~\ref{le:serie}(iii)) and again Lemma~\ref{le:condinte}(i),
we obtain
    \begin{equation*}
    \begin{split}
    \int_t^1\frac{\varrho(s)}{s^2}\,ds&=\int_0^{1-t}\frac{\ho(r)^\eta}{(1-r)^{2-\frac{1}{p}}}\,dr
    \lesssim\frac{\ho(1-t)^{\frac{1}{p}}}{t^{1-\frac{1}{p}}}\int_0^{1-t}\frac{dr}{(1-r)\ho^{\frac{1}{p}-\eta}(r)}\\
    &=\frac{\ho^{\frac{1}{p}}(1-t)}{\ho(1-t)^{\frac{1}{p}-\eta}t^{1-\frac{1}{p}}}\int_0^{1-t}
    \frac{\ho^{\frac{1}{p}-\eta}(1-t)}{\ho^{\frac{1}{p}-\eta}(r)}\frac{dr}{1-r}\\
    &\le\frac{\ho^{\eta}(1-t)t^{\a\left(\frac{1}{p}-\eta\right)}}{t^{1-\frac{1}{p}}}
    \int_0^{1-t}\frac{dr}{(1-r)^{1+\a\left(\frac{1}{p}-\eta\right)}}\\
    &\lesssim\frac{\ho^{\eta}(1-t)}{t^{1-\frac{1}{p}}}=\frac{\varrho(t)}{t},\quad
    0<t<1,
    \end{split}
    \end{equation*}
which is \eqref{b1}.

With the proof of \cite[Theorem~2.1]{BlaSoaJMAA1990} in hand, the
second assertion (ii) can be proved in an analogous manner with
minor modifications.
\end{proof}

If we choose $\om(r)=(1-r^2)^\alpha$, where $-1<p-2<\a<\infty$,
then Theorem~\ref{mainth}(i) and
Proposition~\ref{pr:Lambdafrontera} show that $\hg: A^p_\a \to
A^q_\a$, $q\ge p$, is bounded if and only if $g$ belongs to the
mean Lipschitz space
$\Lambda\left(q,\frac{2+\a}{p}-\frac{1+\a}{q}\right)$.

With respect to the condition that characterizes the bounded
operators $\hg: A^p_\om\to A^q_\om$, when $1<q<p<\infty$ and
$\om\in\R\cap\Mp$, it is worth noticing that
    \begin{equation}\label{qmenorpequiv}
    g'\in H\left(q,s,\ho_{s\left(1-\frac{1}{q}\right)}\right)
    \Leftrightarrow g\in
    H\left(q,s,\ho_{-\frac{s}{q}}\right)
    \end{equation}
by \cite[Theorem~1.1]{PavP}, provided that
$(1-r)^{-\frac{s}{q}}\ho(r)\in\R$. This last requirement may
happen only if $q<p-1$, because
    $$
    (1-r)^{-\frac{s}{q}}\ho(r)=\frac{\ho(r)}{(1-r)^{p-1}}\frac{1}{(1-r)^{1+\frac{p}{p-q}-p}}
    $$
and $\frac{\ho(r)}{(1-r)^{p-1}}$ is essentially increasing. In
particular, the previous argument says that \eqref{qmenorpequiv}
does not hold if $1<p\le 2$. It is also worth noticing that for
the standard weight $\om(r)=(1-r)^\a$ that belongs to $\Mp$, the
equivalence \eqref{qmenorpequiv} is satisfied when
$p-2>\a>\frac{p}{p-q}-2$.

\subsection{Analysis on the Muckenhoupt type condition}

We saw in Theorem~\ref{th:gorro} that the Muckenhoupt type
condition \eqref{Mpcondition} characterizes the boundedness of
both the Hilbert operator $\H$ and the sublinear Hilbert operator
$\widetilde{\H}$ from $L^p_{\widehat{\om}}$ to $A^p_\om$ whenever
$\om$ is regular and satisfies the integral condition \eqref{99},
and further, that the quantity $\mathcal{M}_p(\om)$ is comparable
to the operator norm in both cases. Both integral conditions
\eqref{99} and \eqref{Mpcondition} restrict the behavior of the
inducing weight $\om$ in their own way and thus also affect to the
nature of the spaces $L^p_{\widehat{\om}}$ and $A^p_\om$ as well.
To understand these conditions better, we compare them to the
pointwise behavior of the quotient $\p_\om(r)/(1-r)$ appearing in
the definitions of the regular and rapidly increasing weights.

\begin{lemma}\label{Lemma:limits}
Let $\om$ a continuous radial weight   and $1<p<\infty$.
\begin{itemize}
\item[\rm(i)]If \eqref{99} holds and
    $$
    \liminf_{r\to1^-}\frac{\psi_\om(r)}{1-r}>\frac1{p-1},
    $$
then $\om\in\Mp$. \item[\rm(ii)] If \eqref{99} holds and
    $$
    \lim_{r\to1^-}\frac{\psi_\om(r)}{1-r}=\frac1{p-1},
    $$
then $\om\notin\Mp$.
\item[\rm(iii)] If there exists $r^\star\in(0,1)$ such that
    $$
    \frac{\psi_\om(r)}{1-r}\le\frac1{p-1},\quad r^\star\le r<1,
    $$
then \eqref{99} does not hold.
\end{itemize}
\end{lemma}

\begin{proof}
(i). By the assumption, there exist $d>\frac1{p-1}$ and
$r_0\in(0,1)$ such that $\frac{\psi_\om(r)}{1-r}\ge d$ on
$[r_0,1)$. Therefore the differentiable function
$h_d(r)=\frac{\widehat{\om}(r)}{(1-r)^\frac1d}$ is increasing on
$[r_0,1)$, and hence
    \begin{equation*}
    \widehat{\om}(r)\lesssim
    \left(\frac{1-r}{1-t}\right)^\frac1d\widehat{\om}(t),\quad
    0\le r\le t<1.
    \end{equation*}
It follows that
    \begin{equation*}
    \begin{split}
    \int_{r}^1\widehat{\om}(t)^{-\frac{1}{p-1}}dt
    &\lesssim\widehat{\om}(r)^{-\frac1{p-1}}(1-r)^{\frac1{d(p-1)}}\int_r^1(1-t)^{-\frac1{d(p-1)}}\,dt
    \asymp(1-r)\widehat{\om}(r)^{-\frac{1}{p-1}}.
    \end{split}
    \end{equation*}
Since trivially,
    $$
    \int_{r}^1\widehat{\om}(t)^{-\frac{1}{p-1}}\,dt\ge(1-r)\widehat{\om}(r)^{-\frac{1}{p-1}},
    $$
we deduce $\ho^{-\frac{1}{p-1}}\in\R$, and thus $\om\in\Mp$ by
Lemma~\ref{Lemma:u_p}.

(ii). The assertion follows by the Bernouilli-l'H\^{o}pital
theorem and Lemma~\ref{Lemma:u_p}(i).

(iii). The assumption yields
    \begin{equation*}
    \widehat{\om}(r)\gtrsim
    \left(\frac{1-r}{1-t}\right)^{p-1}\widehat{\om}(t),\quad
    0\le r\le t<1,
    \end{equation*}
and hence
    \begin{equation*}
    \begin{split}
    \int_{r}^1\widehat{\om}(t)^{-\frac{1}{p-1}}dt
    &\gtrsim\widehat{\om}(r)^{-\frac1{p-1}}(1-r)\int_r^1\frac{dt}{1-t}=\infty.
    \end{split}
    \end{equation*}
\end{proof}

It is worth noticing that there there exists regular weights $\om$
such that $\lim_{r\to1^-}\frac{\psi_\om(r)}{1-r}$ does not exist.
The weight $\om$, defined by the identity
    $$
    \int_r^1 \om(s)\,ds=2(1-r)\cos\left(\frac{1}{(1-r)^{1/2}}\right)+16(1-r)^{1/2},
    $$
gives is a concrete example.

The bigger the limit $\lim_{r\to1^-}\frac{\psi_\om(r)}{1-r}$ is
(if it exists), the smaller the space $A^p_\om$ is. Therefore, in
view of Lemma~\ref{Lemma:limits}, Theorem~\ref{mainth}(i) says,
roughly speaking, that the Hilbert operator $\H$ is well defined
and bounded on $A^p_\om$ for $\om\in\R$ whenever the space is
small enough. It is known that the weighted Bergman space
$A^p_\om$ induced by a rapidly increasing weight $\om$ lies closer
to the Hardy space $H^p$ than any classical weighted Bergman space
$A^p_\a$~\cite{PelRat}. Therefore, by the observation above and
results in \cite{GaGiPeSis}, it is natural to expect that if
$\om\in\I$ (and satisfies some local regularity requirement), then
$\hg$ is bounded on $A^p_\om$ if and only if $g$ belongs to the
mean Lipschitz space $\Lambda\left(p,\frac{1}{p}\right)$. The
proof of Theorem~\ref{mainth} with minor modifications show that
$g\in\Lambda\left(p,\frac{1}{p}\right)$ is indeed a necessary
condition for $\hg: A^p_\om\to A^p_\om$ to be bounded when
$\om\in\I$. It is also appropriate to mention that the question of
characterizing the bounded operators $\hg$ on $A^p_\om$ with
$\om\in\I$ is more likely related to the open problem of
describing those $g\in \H(\D)$ such that $\hg: H^p\to H^p$ is
bounded in the case $2<p<\infty$~\cite{GaGiPeSis}.

\bigskip

{\bf{Acknowledgements}} We would like to thank Aristos Siskakis for introducing the generalized Hilbert operator to us.  We would also like to thank Francisco
J- Mart\'{\i}n-Reyes for some helpful discussions on the theory of weights.

\bibliographystyle{elsarticle-num}

\end{document}